\documentclass[pdflatex,sn-mathphys,Numbered]{sn-jnl}

\usepackage{graphicx}
\usepackage{tikz-cd}
\usepackage{amsmath,amssymb,amsfonts,amsthm}
\usepackage{enumitem}
\usepackage{manyfoot}

\usepackage[T1]{fontenc}
\usepackage[final,expansion=false]{microtype}
\theoremstyle{plain}
\newtheorem{theorem}{Theorem}
\newtheorem{proposition}[theorem]{Proposition}
\newtheorem{lemma}[theorem]{Lemma}
\newtheorem{corollary}[theorem]{Corollary}

\theoremstyle{definition}
\newtheorem{definition}[theorem]{Definition}

\theoremstyle{remark}
\newtheorem{remark}[theorem]{Remark}
\newtheorem*{assumption*}{Assumption}

\newcommand{\R}{\mathbb{R}}
\newcommand{\op}{\mathrm{op}}
\newcommand{\Rbar}{\overline{\mathbb{R}}}
\newcommand{\TP}{\mathbb{TP}}
\newcommand{\cat}[1]{\mathcal{#1}}
\DeclareMathOperator{\Nuc}{\mathrm{Nuc}}

\DeclareMathOperator{\pnuc}{\mathbb{P}\mathrm{Nuc}}
\DeclareMathOperator{\im}{\mathrm{Im}}
\DeclareMathOperator{\fix}{\mathrm{Fix}}

\newcommand{\pcat}[1]{\mathbb{P}\mathcal{#1}}
\newcommand{\proj}{Proj}
\DeclareMathOperator{\Cell}{\mathrm{Cell}}

\newcommand{\paren}[1]{\left(#1\right)}

\newcommand{\set}[1]{\left\{#1\right\}}

\begin{document}

\title{Projective metric geometry of tropical nuclei: gap matrices, event loci, and order chambers}

\author[1]{\fnm{Juan Luis} \sur{Gastaldi}}
\email{juan.luis.gastaldi@inf.ethz.ch}
\equalcont{The authors contributed equally to this work.}

\author[2]{\fnm{Samantha} \sur{Jarvis}}
\email{Samantha.jarvis@qc.cuny.edu}
\equalcont{The authors contributed equally to this work.}

\author[3]{\fnm{Thomas} \sur{Seiller}}
\email{thomas.seiller@cnrs.fr}
\equalcont{The authors contributed equally to this work.}

\author[2,4]{\fnm{John} \sur{Terilla}}
\email{jterilla@gc.cuny.edu}
\equalcont{The authors contributed equally to this work.}

\affil[1]{\orgname{ETH Zurich}, \orgaddress{\city{Zurich}, \country{Switzerland}}}
\affil[2]{\orgname{Queens College, City University of New York}, \orgaddress{\city{New York}, \state{NY}, \country{USA}}}
\affil[3]{\orgname{CNRS}, \orgaddress{\city{Paris}, \country{France}}}
\affil[4]{\orgname{The Graduate Center, City University of New York}, \orgaddress{\city{New York}, \state{NY}, \country{USA}}}

\abstract{The tropical row span and column span of a real matrix are,
	from the polyhedral point of view, different objects living in different
	ambient spaces.  These polytopes are known to be combinatorially isomorphic
	as polyhedral complexes; we prove that they are isometric
	under a Hilbert projective metric.  We show that this isometry, along with a considerable
	amount of additional metric and polyhedral structure, is a direct consequence
	of a single categorical construction: the Isbell nucleus of the matrix,
	viewed as a profunctor enriched over the extended reals.

	The projective nucleus carries two canonical structures
	inherited from enrichment.
	The first is a Hilbert projective metric, with respect to which the
	Isbell conjugate maps are mutually inverse isometries---this is the
	Isometry Theorem.
	The second is a polyhedral cell decomposition cut out by the Isbell
	inequalities, recovering the type decomposition of tropical convexity.

	These two structures are linked pointwise by the \emph{gap matrix}.  
	The Events Theorem identifies each positive entry of the gap matrix with
	the exact projective distance to the locus where the corresponding
	inequality becomes tight:
	algebraic slack in the Isbell inequalities equals geometric distance
	to the cell walls.
	Thresholding the gap matrix at successive radii produces a constructible sheaf
	of formal concept lattice towers, extracting discrete algebraic structure from
	the continuous geometry at each point.

	In the square case there is generically a unique
	full-dimensional cell.  The Centering Theorem identifies its Chebyshev
	center---the point maximally insulated from all cell walls---and shows that
	the optimal radius equals the minimum directed cycle mean of an associated
	digraph, connecting the projective geometry of the nucleus to the classical
	theory of optimal assignments.
}


\maketitle

\section{Introduction}\label{sec:introduction}
A real $m\times n$ matrix $M$ gives rise to two tropical polytopes
in tropical projective spaces: a row span in $\TP^{n-1}$ and a
column span in $\TP^{m-1}$.
Each carries a polyhedral decomposition into cells according to
combinatorial type.
A foundational result of Develin and
Sturmfels~\cite{develinSturmfels2004tropical} is that these two
polytopes are isomorphic as polyhedral complexes: there is a
canonical bijection between their cells that preserves combinatorial
type.
Their proof is a direct polyhedral argument, relying on the structure
of the type decomposition.
The starting point of this paper is the observation that a stronger
fact holds and admits a conceptual explanation: the row and column
spans are not merely combinatorially isomorphic but also isometric
under the Hilbert projective metric, and this isometry is an
immediate consequence of the fact that both are projections of a
single intrinsic object, the \emph{Isbell nucleus} of~$M$.

The Isbell nucleus is defined for any $\Rbar$-enriched profunctor
$M\colon\cat C^{\op}\otimes\cat D\to\Rbar$, where $\Rbar$ is the monoidal poset
$\Rbar=([-\infty,\infty],\le,+)$.
The Isbell conjugates $M^*$ and $M_*$ form an adjunction between
$\Rbar$-enriched presheaves on $\cat C$ and copresheaves on
$\cat D$:
\[
	M^*f(d)=\min_{c\in\cat C}\bigl(M(c,d)-f(c)\bigr),
	\qquad
	M_*g(c)=\min_{d\in\cat D}\bigl(M(c,d)-g(d)\bigr).
\]
The nucleus $\Nuc(M)$ is the fixed-point locus, the set of
pairs $(f,g)$ with $M^*f=g$ and $M_*g=f$.
When $\cat C$ and $\cat D$ are finite sets and $M$ is a real matrix,
the projectivization $\pnuc(M)$ is a compact polyhedral space
carrying two canonical structures: a \emph{projective metric} coming
from the enriched hom, and a \emph{polyhedral cell decomposition}
coming from the Isbell inequalities $f(c)+g(d)\le M(c,d)$.
The cell decomposition recovers the type decomposition of tropical
convexity: a point is classified by the incidences $(c,d)$ for which
equality holds, which we call \emph{witness pairs}.
Both structures are invariant under external gauge
transformations---reweighting $M$ by row and column
potentials---since these act by isometries.
The nucleus is the intrinsic geometric object; the matrix $M$ is
just a coordinate presentation.
In particular, $\pnuc(M)$ maps isometrically onto both the tropical
row span and the tropical column span: there is no mystery about why
the two spans are isometric, because they are projections of a single
intrinsic space.

The central new tool is the \emph{gap matrix} defined at each nucleus point $(f,g)$ 
by
$\delta^{(f,g)}(c,d)=M(c,d)-f(c)-g(d)$.  It is a non-negative matrix with at least one
zero in every row and column.
Its zero entries record the witnesses and hence the combinatorial cell.
Its positive entries, which measure the slack in each inequality, turn
out to have a sharp geometric meaning.
The three main results of the paper are the following.

\paragraph{Isometry Theorem (Theorem~\ref{thm:projective-isometries}).}
	      The projective Isbell maps $M^*$ and $M_*$ restrict to mutually
	      inverse isometries between the presheaf and copresheaf
	      realizations of $\pnuc(M)$ for the Hilbert projective metric.
	      The tropical row and column spans 
	      are here revealed as two projections of a single intrinsic
	      space.

\paragraph{Events Theorem (Theorem~\ref{thm:events}).}
	      Each positive entry of the gap matrix equals the exact
	      projective distance to the corresponding event locus:
	      \[
		      d_{\pnuc}\paren{(f,g),\mathcal E_{c,d}}=\delta^{(f,g)}(c,d),
	      \]
	      where $\mathcal E_{c,d}$ is the locus on which $(c,d)$ is
	      a witness pair.
	      In other words, algebraic slack in the Isbell inequalities
	      \emph{is} geometric distance to the cell walls.
	      This identity reflects a rigidity special to the Hilbert metric
	      and the linear structure of the Isbell conditions:
	      in a generic polyhedral-metric setting, constraint slack and
	      boundary distance can differ by arbitrary distortion factors.

\paragraph{Centering Theorem (Theorem~\ref{thm:centering}).}
	      In the square case $|\cat C|=|\cat D|=n$, there is generically one
	      full-dimensional cell.
	      The Chebyshev center of this cell---the point maximizing the
	      minimum distance to all cell walls---has optimal radius equal to the
	      minimum directed cycle mean of an associated weighted digraph.
	      At the center, the smallest positive gap is achieved with multiplicity
	      at least $n$ (generically exactly~$n$), identifying $n$ equidistant
	      event loci.
	      The minimum cycle mean is computable in $O(n^3)$ time by Karp's
	      algorithm~\cite{Karp1978}, connecting
	      the projective geometry of the nucleus to the classical theory of
	      optimal assignment~\cite{Kuhn1955,BurkardDellAmicoMartello2009}.

\paragraph{Further structures.}
The Events Theorem has further structural consequences.
Thresholding the gap matrix at a value $\varepsilon > 0$ records which
event loci lie within projective distance $\varepsilon$ of a given
point.
The resulting Boolean relation is itself a profunctor, and its Isbell
nucleus is a finite lattice: a formal concept lattice in the sense
of Wille.
As $\varepsilon$ increases, new witness pairs enter and the lattice grows. Since the gap matrix has finitely many distinct values, this growth factors through a finite tower of concept lattices, one for each distinct gap value. Wall-crossing between consecutive floors is governed by canonical mergers.
The tower depends on the basepoint $(f,g)$, but only through the
ordering of the gap entries: within each \emph{order chamber}---a
region where the ordering is constant---the tower is invariant.
Moving to a face of the chamber complex merges consecutive floors,
producing canonical specialization maps.  The order chambers give a
refinement of the polyhedral decomposition defined by
witnesses and appears to be new in tropical geometry.

This pointed thresholding procedure deserves a moment of emphasis.
The naive operation of thresholding the matrix $M$ directly ignores the
geometry of the nucleus and produces no meaningful structure, as far as we can see.
But thresholding the gap matrix $\delta^{(f,g)}$---which depends on one's
position in the nucleus---extracts a principled family of discrete algebraic
structures (formal concept lattices with joins, meets, and Galois
connections) from the continuous projective geometry, organized into a
constructible sheaf over the polyhedral order-chamber complex.

\paragraph{A running example.}
To keep the discussion concrete, we repeatedly return to the following matrix.
Let $C=\{c_0,c_1,c_2\}$ and $D=\{d_1,d_2,d_3,d_4\}$ and set
\[
	M=
	\begin{bmatrix}
		0.7 & 1.5  & 1.7 & -1.3 \\
		1.2 & 2.6  & 0.1 & 2.2  \\
		2.0 & -1.6 & 2.0 & -2.9
	\end{bmatrix}.
\]
In this case $\pnuc(M)$ is a two-dimensional polyhedral complex.
Work in the gauge slice $f(c_0)=0$ and consider the point $(f,g)$ with
\[
	f=(0,0,0),\qquad g=(0.7,-1.6,0.1,-2.9).
\]
Its gap matrix is
\[
	\delta^{(f,g)}=
	\begin{bmatrix}
		0   & 3.1 & 1.6 & 1.6 \\
		0.5 & 4.2 & 0   & 5.1 \\
		1.3 & 0   & 1.9 & 0
	\end{bmatrix},
\]
whose positive values satisfy
\[
	0<0.5<1.3<1.6=1.6<1.9<3.1<4.2<5.1.
\]
The zero pattern
$\{(c_0,d_1),(c_2,d_2),(c_1,d_3),(c_2,d_4)\}$
determines the witness cell containing $(f,g)$, while the positive entries are---by the Events Theorem---the exact projective distances from $(f,g)$ to the surrounding cell walls.
The tie at $1.6$ is the value at which the order of the gap entries changes, producing a wall in the order-chamber decomposition.
Figure~\ref{fig:cellular_events} illustrates this gap-value-as-distance phenomenon.

\begin{figure}[t]
	\centering
	\includegraphics[width=0.90\linewidth]{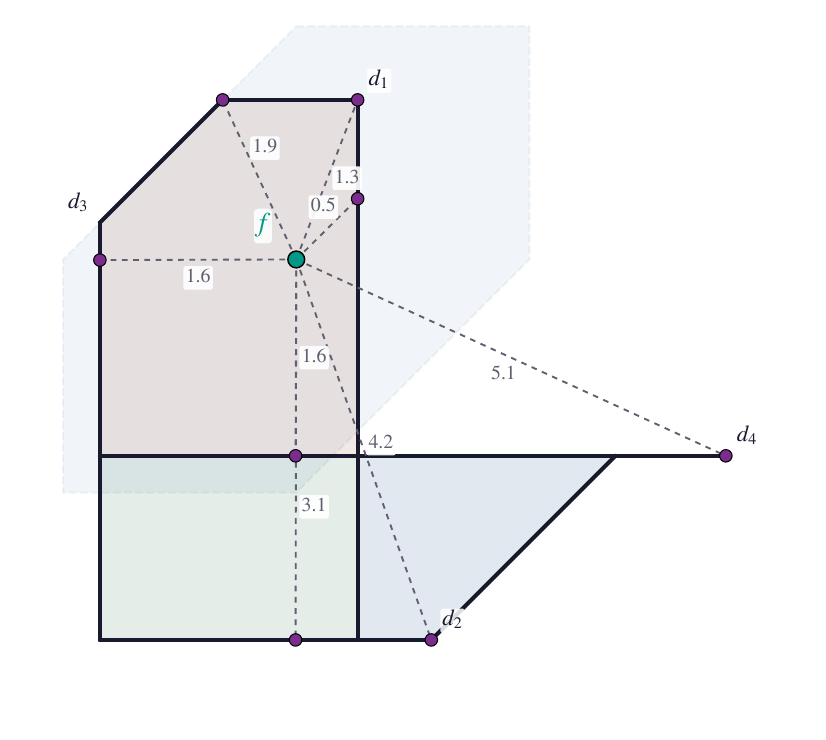}
	\caption{The running $3\times 4$ example in an affine chart of $\pnuc(M)$.
		The green point is the basepoint $(f,g)$, and the shaded regions are the $2$-cells of the witness decomposition.
		For each positive gap value, the corresponding marked event point lies on the first event locus encountered at that radius, as indicated by the dashed segment from $(f,g)$.
		The repeated value $1.6$ is the tie that later becomes the order wall between two adjacent chamber refinements.}
	\label{fig:cellular_events}
\end{figure}

\subsection{Relation to existing work}
On the tropical side, Develin and Sturmfels introduced the polyhedral theory of tropical convexity and its decomposition into combinatorial types~\cite{develinSturmfels2004tropical}; see also \cite{MaclaganSturmfels2015ITG} for background.
The half-space and distance formulas that underlie our Events Theorem are close in spirit to work of Gaubert and Katz on max-plus convexity and tropical half-spaces~\cite{GaubertKatz2006MaxPlusConvexGeometry,GaubertKatz2011MinimalHalfspaces}; compare also \cite{Nitica01062007}.
Gaubert and Sergeev~\cite{GaubertSergeev2013CyclicProjectors} study cyclic projectors and separation in idempotent convex geometry, using the Hilbert projective metric in the tropical setting; our $1$-Lipschitz and isometry results for the Isbell maps can be seen as complements to their spectral approach.

On the categorical side, the idea of treating generalized metric spaces as categories enriched over an ordered monoid goes back to Lawvere~\cite{lawvere1973metric}.
Isbell conjugacy originates in Isbell's paper on adequate subcategories~\cite{Isbell1960Adequate}, and Avery and Leinster~\cite{averyLeinster2021isbell} give a modern treatment over an arbitrary base.
In the setting of Lawvere metric spaces, Willerton developed the Isbell completion in detail~\cite{willerton2013tight,willerton2014galois,willerton2015legendre}, relating it to the Legendre--Fenchel transform.
Via the $c$-transform, the same formalism also connects to optimal transport~\cite{ambrosioGigli2013usersGuide}.

The bridge from enriched-category theory to tropical and directed-metric phenomena has been explored from several directions.
Elliott and Fujii observed that Isbell-type nuclei provide a natural home for tropical polytopes~\cite{elliott2017fuzzy,fujii2019enriched}.
Bradley, Terilla, and Vlassopoulos use enrichment over $[0,1]\cong[0,\infty]$ to organize linguistic structure~\cite{BradleyTerillaVlassopoulos2022EnrichedLanguage}, while Gaubert and Vlassopoulos develop related directed-metric and tropical-polyhedral ideas in the setting of large language models~\cite{GaubertVlassopoulos2024DirectedMetric}.
Recent work of Bradley and Vigneaux studies magnitude and magnitude homology for categories of texts enriched by language-model probabilities~\cite{BradleyVigneaux2025MagnitudeTexts}.
Background and motivation for these applications are reviewed in \cite{BradleyGastaldiTerilla2024}.

The Centering Theorem connects to the combinatorial optimization literature.
The minimum cycle mean of a weighted digraph is a classical invariant studied by Karp~\cite{Karp1978}, and the
assignment-problem duality underlying the Chebyshev LP is closely related to the Hungarian method~\cite{Kuhn1955,BurkardDellAmicoMartello2009}; see \cite[Ch.~17]{Schrijver2003} for a textbook treatment.
In a related but distinct direction, Akian, Gaubert, Qi, and Saadi~\cite{AkianGaubertQiSaadi2023TropicalRegression} prove that the inner radius of a Hilbert ball inscribed in a tropical polyhedron equals the value of a mean payoff game; our Centering Theorem can be viewed as a pointwise refinement of this circle of ideas, identifying the Chebyshev radius of a specific cell within the nucleus with the minimum cycle mean of an explicit digraph derived from the optimal assignment.

These constructions fit into a broader program in which nuclei acquire additional compatible structure.
In particular, nuclei arising from profunctors compatible with monoidal data carry further operations relevant to linear realizability; see \cite{Jarvis2025NucleusProfunctor,seiller-hdr,GastaldiJarvisSeillerTerillaLinearRealizability}.
The finite real case studied here exhibits the projective metric geometry, witness cells, event loci, order chambers, the associated constructible sheaf of lattice towers, and the Chebyshev centering of full-dimensional cells.

Our conventions differ slightly from some of the literature, and these differences matter for the geometry.
We therefore keep Section~\ref{sec:nucleus} self-contained, but restrict it to the categorical material used later.
The cocompletion formulas behind the Yoneda-density statements are standard; see~\cite{Kelly1982}.

\medskip
\noindent
\textbf{Organization of the paper.}
Section~\ref{sec:nucleus} fixes conventions for $\Rbar$-enriched categories, profunctors, and the Isbell adjunction.
Section~\ref{sec:geometry} develops the projective metric geometry and proves the Isometry Theorem.
Section~\ref{subsec:polyhedral_pnuc} develops the witness polyhedral structure, proves the Events Theorem, introduces order chambers, and shows how pointed thresholding of the gap matrix assembles into chamberwise concept-lattice towers with canonical face-specialization maps.
Section~\ref{sec:centering} specializes to the square case and proves the Centering Theorem.

\section{Isbell duality over the extended reals}\label{sec:nucleus}

We fix conventions for $\Rbar$-enriched categories, the Isbell adjunction associated to a profunctor
$M\colon\cat C\nrightarrow\cat D$, and its nucleus $\Nuc(M)$.
The weighted-colimit formulas behind the density identities are standard; see~\cite{Kelly1982}.
\subsection{The arithmetic of $\Rbar$}\label{subsec:Rbar}

We begin with the real numbers ordered by $\le$.  Viewed as a category, its limits are infima and its colimits are suprema.
Adjoining $\pm\infty$ makes $\Rbar=\R\cup \{\pm \infty\}$ into a poset category that is complete and cocomplete:
infema and suprema of arbitrary subsets exist.
Addition of real numbers defines a symmetric monoidal structure which extends to
$(\Rbar,\le,+,0)$ by declaring $-\infty$ absorbing:
$-\infty + y = -\infty$ for all $y\in\Rbar$, including
$y=+\infty$.
This convention is dictated by the requirement that each translation
$x\mapsto x+y$ preserve colimits (suprema) and hence have a right adjoint:
any other extension of addition to~$\pm\infty$ would violate this.
The right adjoint is \emph{residuation}
$z - y := [y,z]$, characterized by $x + y \le z \iff x \le z - y$
and given explicitly~by
\begin{equation}\label{eq:minus}
  z - y = \sup\set{x\in\Rbar \mid x + y \le z}.
\end{equation}
On finite reals this is ordinary subtraction.
At the boundary one has $\infty - \infty = \infty$ and
$-\infty - (-\infty) = \infty$; in particular, subtracting $-\infty$
is not the same as adding $+\infty$.
Three properties are used throughout: residuation $z\mapsto z-y$
is monotone and preserves infima (being a right adjoint); the reverse map $z\mapsto x-z$
is antitone; and $-\infty$ is absorbing for addition while $\infty$
is absorbing for residuation.

We use $\Rbar$ rather than Lawvere's
$([0,\infty],\ge,+,0)$ because both infinite values play a
geometric role.
Subsets $A\subseteq C\times\R$ arise naturally in our
applications, and they determine presheaves $C\to\Rbar$ by
$c\mapsto\sup\set{r\in\R\mid(c,r)\in A}$: one needs $f(c)=-\infty$
when $A$ contains no point over~$c$ and $f(c)=+\infty$ when it
contains all of them.
More immediately, in Lawvere's base $([0,\infty],\ge,+,0)$
the monoidal unit $0$ is the top element, so residuation is truncated: $z - y = \max(z-y, 0)$. 
In $\Rbar$ the unit $0$ sits in the interior, residuation is ordinary subtraction, 
and the resulting translation action on (co)presheaves produces 
affine geometry—then projective geometry after quotienting by constant shifts.

\subsection{Categories, functors, presheaves, and profunctors}\label{subsec:yoneda}
An \emph{$\Rbar$-category} $\cat C$ consists of a set
  $\mathrm{Ob}(\cat C)$ and hom-values
  $\cat C(c,c')\in\Rbar$ satisfying $\cat C(c,c)=0$ (identities) and
  $\cat C(c,c')+\cat C(c',c'')\le\cat C(c,c'')$ (composition).
  The \emph{opposite} $\cat C^{\op}$ reverses the hom-values:
  $\cat C^{\op}(c,c')=\cat C(c',c)$.
  An \emph{$\Rbar$-functor} $F\colon\cat C\to\cat D$ is a map on
  objects satisfying $\cat C(c,c')\le\cat D(Fc,Fc')$; the enriched
  hom between functors is $[\cat D,\Rbar](F,G)=\inf_{c\in\cat C}\cat D(Fc,Gc)$, 
  making $[\cat D,\Rbar]$, the set of $\Rbar$-functors from $\cat{D}$ to $\Rbar$, into an $\Rbar$-category.
  If it is not ambiguous, we abbreviate $[\cat D,\Rbar](F,G)$ by $[F,G].$

A \emph{presheaf} on $\cat C$ is an $\Rbar$-functor $f\colon\cat C^{\op}\to\Rbar$;
a \emph{copresheaf} on $\cat D$ is an $\Rbar$-functor $g\colon\cat D\to\Rbar$.
We regard copresheaves as objects of $[\cat D,\Rbar]^{\op}$ (same underlying functions, reversed pointwise order),
so that the Isbell conjugates below become $\Rbar$-functors.
The enriched homs are computed pointwise:
\[
	[\cat C,\Rbar](f,f')=\inf_{c\in\cat C}\paren{f'(c)-f(c)},\qquad
	[\cat D,\Rbar]^{\op}(g,g')=\inf_{d\in\cat D}\paren{g(d)-g'(d)}.
\]
In particular, the underlying order on presheaves is the pointwise order, while copresheaves in
$[\cat D,\Rbar]^{\op}$ carry the opposite of the pointwise order.

The enriched Yoneda lemma says that
$[\cat C(-,c),f]=f(c)$ and $[g,\cat D(d,-)]_{[\cat D,\Rbar]^{\op}}=g(d)$.  In particular,
$[\cat C(-,c),\cat C(-,c')]=\cat C(c,c')$ and $[\cat D(d,-),\cat D(d',-)]=\cat D(d',d)$ so 
the Yoneda map $y\colon\cat C\to[\cat C^{\op},\Rbar]$, $c\mapsto\cat C(-,c)$,
and its copresheaf analogue $\cat D\to[\cat D,\Rbar]^{\op}$, $d\mapsto\cat D(d,-)$,
are fully faithful embeddings.
Every presheaf and copresheaf is a pointwise supremum of representables:
\[
	f(x)=\sup_{c\in\cat C}\paren{f(c)+\cat C(x,c)},
	\qquad
	g(x)=\sup_{d\in\cat D}\paren{g(d)+\cat D(d,x)};
\]
for the weighted-colimit formulation, see~\cite{Kelly1982}.

A \emph{profunctor} $M\colon\cat C\nrightarrow\cat D$ is an $\Rbar$-functor
$M\colon\cat C^{\op}\otimes\cat D\to\Rbar$,
where $\cat C^{\op}\otimes\cat D$ has objects $\mathrm{Ob}(\cat C)\times\mathrm{Ob}(\cat D)$ and
$(\cat C^{\op}\otimes\cat D)((c,d),(c',d'))=\cat C(c',c)+\cat D(d,d')$.
Any set $S$ determines a discrete $\Rbar$-category with $S(s,s')=0$ if $s=s'$ and $-\infty$ otherwise;
for discrete $\cat C$ and $\cat D$, every function $M\colon\cat C\times\cat D\to\Rbar$ defines a profunctor.

\subsection{Isbell duality and the nucleus}\label{subsec:nucleus}

Let $M\colon\cat C\nrightarrow\cat D$ be a profunctor.
For each $d\in\cat D$ the column $M(-,d)$ is a presheaf on $\cat C$,
and for each $c\in\cat C$ the row $M(c,-)$ is a copresheaf on $\cat D$.
The Isbell conjugates extend these assignments to arbitrary (co)presheaves.

\begin{definition}\label{def:Isbell-conjugates}
	The \emph{Isbell conjugates} of $M$ are
	\[
		\begin{aligned}
			M^*\colon [\cat C^{\op},\Rbar] & \to [\cat D,\Rbar]^{\op}, \\
			M_*\colon [\cat D,\Rbar]^{\op} & \to [\cat C^{\op},\Rbar],
		\end{aligned}
	\]
	defined by
	\begin{align}
		(M^*f)(d) & := \inf_{c\in\cat C}\paren{M(c,d)-f(c)},\label{eq:Mupperstar} \\
		(M_*g)(c) & := \inf_{d\in\cat D}\paren{M(c,d)-g(d)}.\label{eq:Mlowerstar}
	\end{align}
\end{definition}

\begin{proposition}\label{prop:Isbell-adjunction}
	The maps $M^*$ and $M_*$ are $\Rbar$-functors and satisfy $M^*\dashv M_*$, i.e.,
	\[
		[\cat D,\Rbar]^{\op}(M^*f,g)=[\cat C^{\op},\Rbar](f,M_*g)
	\]
	for all presheaves $f$ and copresheaves $g$.
	Moreover, $M^*$ and $M_*$ are order-reversing for the pointwise order.
\end{proposition}

\begin{proof}
	Functoriality of $M^*$: if $\delta=[f,f']=\inf_{c}(f'(c)-f(c))$ then $f(c)\le f'(c)-\delta$ for all $c$, and taking
	$\inf_c$ of $M(c,d)-f'(c)+\delta\le M(c,d)-f(c)$ yields
	$[f,f']\le[\cat D,\Rbar]^{\op}(M^*f,M^*f')$; similarly for~$M_*$.

	For the adjunction, the key step is that residuation by $g(d)$ preserves infima:
	\begin{align*}
		[\cat D,\Rbar]^{\op}(M^*f,g)
		 & =\inf_d\paren{\inf_c\paren{M(c,d)-f(c)}-g(d)} \\
		 & =\inf_d\inf_c\paren{M(c,d)-f(c)-g(d)}         \\
		 & =\inf_c\paren{\inf_d\paren{M(c,d)-g(d)}-f(c)}
		=[\cat C^{\op},\Rbar](f,M_*g).
	\end{align*}
	Antitonicity: if $f\le f'$ then $M(c,d)-f'(c)\le M(c,d)-f(c)$ by antitonicity of residuation,
	so $(M^*f')(d)\le(M^*f)(d)$; the argument for $M_*$ is the same.
\end{proof}

Since $M^*$ and $M_*$ are each antitone, the composites
$M_*M^*$ and $M^*M_*$ are monotone.
The adjunction $M^*\dashv M_*$ gives $f\le M_*M^*f$ for every
presheaf~$f$: indeed, the adjunction identity applied with $g=M^*f$
reads $[M^*f,M^*f]=[f,M_*M^*f]$, and the left side is $\ge 0$,
so $f\le M_*M^*f$.
Dually $g\le M^*M_*g$.
Monotonicity and expansion together imply idempotence: applying
$M_*M^*$ to $f\le M_*M^*f$ gives
$M_*M^*f\le (M_*M^*)^2f$
(the expanding direction), and applying the antitone $M^*$ to the
same inequality gives
$M^*M_*M^*f\ge M^*f$, then applying the antitone
$M_*$ reverses again:
$(M_*M^*)^2f=M_*M^*M_*M^*f\le M_*M^*f$.
So $M_*M^*$ and $M^*M_*$ are closure operators:
\begin{equation}\label{eq:closure}
	\begin{aligned}
		f & \le M_*M^*f, & (M_*M^*)^2 & =M_*M^*, \\
		g & \le M^*M_*g, & (M^*M_*)^2 & =M^*M_*.
	\end{aligned}
\end{equation}
In other words, each closure operator is a projection: it expands
its input to the nearest fixed point, and once there, does nothing.

\begin{definition}\label{def:nucleus}
	The \emph{nucleus} of $M$ is the $\Rbar$-category
	\[
		\Nuc(M)=\set{(f,g)\mid g=M^*f,\ f=M_*g},
	\]
	with hom-values $\Nuc(M)((f,g),(f',g'))=[f,f']=[\cat D,\Rbar]^{\op}(g,g')$,
	the equality holding by the adjunction.
\end{definition}

\begin{proposition}\label{prop:nucleus-fixedpoints}
	The projection $(f,g)\mapsto f$ identifies $\Nuc(M)$ with
	$\fix(M_*M^*)$, and $(f,g)\mapsto g$ identifies it
	with $\fix(M^*M_*)$.  Moreover,
	$\fix(M_*M^*)=\mathrm{im}(M_*M^*)
	=\mathrm{im}(M_*)$ and
	$\fix(M^*M_*)=\mathrm{im}(M^*M_*)
	=\mathrm{im}(M^*)$.
\end{proposition}

\begin{proof}
	The first identification is immediate:
	$(f,g)\in\Nuc(M)$ iff $f = M_*g = M_*M^*f$;
	dually for the second.

	For any closure operator, fixed points and image coincide:
	every fixed point is in the image (it is its own closure),
	and every element of the image is fixed by idempotence.
	So $\fix(M_*M^*)=\mathrm{im}(M_*M^*)$.

	It remains to show that $M_*$ already lands among the
	fixed points of $M_*M^*$, i.e.\ $M_*M^*M_* = M_*$.
	The expansion $g\le M^*M_*g$
	and the antitonicity of $M_*$ give $M_*M^*M_*g\le M_*g$.
	The reverse inequality is the expansion
	$M_*g\le M_*M^*M_*g$.
	Together: $M_*M^*M_*=M_*$, so $M_*$ lands in
	$\fix(M_*M^*)$.
	Dually $M^*M_*M^*=M^*$.
\end{proof}

\begin{corollary}\label{cor:fiber-max}
	A presheaf $f$ is a fixed point of $M_*M^*$ if and only if it is
	the largest presheaf with the same $M^*$-image:
	$M^*h=M^*f\Rightarrow h\le f$.
	Dually for copresheaves and~$M_*$.
\end{corollary}

\begin{proof}
	If $f=M_*M^*f$ and $M^*h=M^*f$, then
	$h\le M_*M^*h=M_*M^*f=f$ by~\eqref{eq:closure}.
	Conversely, applying this to $h=M_*M^*f$ gives $M_*M^*f\le f$,
	and~\eqref{eq:closure} gives the reverse inequality.
\end{proof}

\section{Projective metric geometry}\label{sec:geometry}

The nucleus constructed in Section~\ref{sec:nucleus} is an
$\Rbar$-enriched category, and the enrichment determines a
canonical projective metric.
The enriched hom $[f,f']=\inf_c(f'(c)-f(c))$ is a directed
distance: $[f,f']\ge 0$ iff $f\le f'$, and enriched composition
gives the directed triangle inequality.
Symmetrizing introduces a kernel: $-[f,f']-[f',f]=0$ exactly when
$f-f'$ is constant, so the symmetrized distance descends to a
genuine metric on translation classes.
For real-valued presheaves, this metric has a simple form:
\[
	d_{\cat C}([f],[f'])
	= \max_{c}\paren{f(c)-f'(c)} - \min_{c}\paren{f(c)-f'(c)},
\]
the oscillation of the difference---the tropical form of a 
Hilbert projective metric.  
The Isbell transforms, being equivariant for constant translation,
respect this projectivization.
For general $\Rbar$-valued presheaves, some care is needed at
$\pm\infty$: first to handle fixed points of the translation
action, and then because residuation at $\pm\infty$ does not
behave like ordinary subtraction.
We handle this in \S\ref{subsec:hilbert-oscillation} below.

\subsection{Finite index sets and the Isbell transforms}\label{subsec:finite-index}

For the geometric constructions below it is convenient to work with
finite, discrete $\Rbar$-categories.
Thus, for the remainder of this section, $\cat C$ and $\cat D$ are
finite sets (regarded as discrete $\Rbar$-categories), and a
profunctor $M\colon \cat C\nrightarrow \cat D$ is simply a function
\[
	M\colon \cat C\times \cat D\to \Rbar.
\]
In this setting, presheaves on $\cat C$ and copresheaves on $\cat D$
are just functions $\cat C\to\Rbar$ and $\cat D\to\Rbar$, and the
infima in Definition~\ref{def:Isbell-conjugates} are minima.
Accordingly,
\begin{equation}\label{eq:Isbell-min-formulas}
	\begin{aligned}
		(M^*f)(d) & = \min_{c\in\cat C}\paren{M(c,d)-f(c)}, \\
		(M_*g)(c) & = \min_{d\in\cat D}\paren{M(c,d)-g(d)},
	\end{aligned}
\end{equation}
where $z-y$ denotes residuation in $\Rbar$
(cf.\ \eqref{eq:minus}).

\begin{lemma}\label{lem:translation-equivariance}
	For any finite constant $\lambda\in\R$ one has
	\begin{equation}\label{eq:translation-equivariance}
		\begin{aligned}
			M^*(f+\lambda) & = M^*f-\lambda, \\
			M_*(g-\lambda) & = M_*g+\lambda,
		\end{aligned}
	\end{equation}
	where $\lambda$ denotes the constant function on $\cat C$
	or $\cat D$.
\end{lemma}

\begin{proof}
	For the first identity, subtracting $\lambda$ from each term
	inside the minimum gives
	\begin{align*}
		(M^*(f+\lambda))(d)
		 & =\min_{c}\paren{M(c,d)-f(c)-\lambda}         \\
		 & =\paren{\min_{c}\paren{M(c,d)-f(c)}}-\lambda \\
		 & =(M^*f)(d)-\lambda.
	\end{align*}
	The second identity is analogous.
\end{proof}

The equivariance \eqref{eq:translation-equivariance} is the
algebraic shadow of a projective symmetry: if $(f,g)$ satisfies
$g=M^*f$ and $f=M_*g$, then so does $(f+\lambda,g-\lambda)$.

\subsection{Projective (co)presheaves and the
Hilbert--oscillation metric}\label{subsec:hilbert-oscillation}

We now isolate the locus on which translation by constants acts
freely.

\begin{definition}\label{def:finite-somewhere}
	Define the \emph{finite somewhere} presheaves
	$[\cat C^{\op},\Rbar]_{\mathrm{fs}}$ to be the full subcategory
	of presheaves $f\colon \cat C\to\Rbar$ for which $f(c)\in\R$ for
	at least one $c\in\cat C$.
	Define $[\cat D,\Rbar]^{\op}_{\mathrm{fs}}$ similarly for
	copresheaves on $\cat D$.
\end{definition}

On these full subcategories, $(\R,+)$ acts freely by constant
translation $f\mapsto f+\lambda$.

\begin{definition}\label{def:projective-presheaves}
	The \emph{projective presheaf space} of $\cat C$ is the quotient
	\[
		\pcat C := [\cat C^{\op},\Rbar]_{\mathrm{fs}}/\R.
	\]
	Let $[f]\in\pcat C$ denote the translation class of $f$.
	Similarly, the \emph{projective copresheaf space} of $\cat D$ is
	\[
		\pcat D := [\cat D,\Rbar]^{\op}_{\mathrm{fs}}/\R.
	\]
\end{definition}

For $\lambda\in\R$ and presheaves $f,f'$, the enriched hom
satisfies $[f+\lambda,f']=[f,f']-\lambda$ and
$[f,f'+\lambda]=[f,f']+\lambda$.
In particular, the symmetrized quantity $-[f,f']-[f',f]$ is
$\R$-invariant.
This is the tropical analogue of Hilbert's projective metric: it
measures only the oscillation of the difference.

When $f$ and $f'$ are real-valued (i.e.\ $f,f'\colon\cat C\to\R$),
the projective distance is simply
\[
	d_{\cat C}([f],[f'])
	= \max_{c}\paren{f(c)-f'(c)} - \min_{c}\paren{f(c)-f'(c)},
\]
the oscillation of the difference.
For general $\Rbar$-valued presheaves, some care is needed.
The residuation $f(c)-f'(c)$ is not antisymmetric when both values
are $+\infty$ or both are $-\infty$: one has
$\infty-\infty=\infty$ and $-\infty-(-\infty)=\infty$, so
subtracting $-\infty$ is not the same as adding $+\infty$.
We isolate the offending indices.

\begin{definition}\label{def:oscillation-distance}
	Let $f,f'\in[\cat C^{\op},\Rbar]_{\mathrm{fs}}$.  Set
	\[
		S(f,f'):=\set{c\in\cat C\mid
		f(c)=f'(c)=+\infty\text{ or }f(c)=f'(c)=-\infty}.
	\]
	These are the indices at which residuation fails to be
	antisymmetric.
	Define the \emph{projective distance} between
	$[f],[f']\in\pcat C$ by
	\begin{equation}\label{eq:projective-metric}
		d_{\cat C}([f],[f'])
		:=
		\begin{cases}
			\begin{aligned}
				 & \sup\limits_{c\notin S(f,f')}
				\paren{f(c)-f'(c)}                     \\
				 & \quad -\inf\limits_{c\notin S(f,f')}
				\paren{f(c)-f'(c)}
			\end{aligned}
			        & \text{if both extrema lie in $\R$}, \\
			+\infty & \text{otherwise}.
		\end{cases}
	\end{equation}
	Define $d_{\cat D}$ on $\pcat D$ analogously.
\end{definition}

\begin{proposition}\label{prop:projective-metric}
	The function $d_{\cat C}$ is an extended metric on $\pcat C$.
	Moreover, whenever $d_{\cat C}([f],[f'])<\infty$ one has the
	identity
	\begin{equation}\label{eq:d-as-symmetrized-hom}
		d_{\cat C}([f],[f'])=-[f,f']-[f',f],
	\end{equation}
	where $[f,f']$ is the enriched hom in $[\cat C^{\op},\Rbar]$.
\end{proposition}

\begin{proof}
	Well-definedness on $\pcat C$ is immediate: translating $f$ by a
	finite constant $\lambda$ shifts both $\sup(f-f')$ and
	$\inf(f-f')$ by $\lambda$, so their difference is unchanged, and
	$S(f,f')$ is unaffected since $\pm\infty+\lambda=\pm\infty$.
	Symmetry and nonnegativity are clear.
	If $d_{\cat C}([f],[f'])=0$ then $f(c)-f'(c)$ is a single real
	constant on $\cat C\setminus S(f,f')$ and
	$f(c)=f'(c)\in\set{\pm\infty}$ on $S(f,f')$, so $[f]=[f']$.

	Now assume $d_{\cat C}([f],[f'])<\infty$.
	Since $+\infty$ is absorbing for residuation, every
	$c\in S(f,f')$ contributes $f'(c)-f(c)=+\infty$ to the enriched
	hom $[f,f']$ and $f(c)-f'(c)=+\infty$ to $[f',f]$.
	These values cannot achieve either infimum (which is finite by
	hypothesis), so they drop out and
	\[
		-[f,f']-[f',f]
		= \sup_{c}\paren{f(c)-f'(c)}
		  - \inf_{c}\paren{f(c)-f'(c)}
		= d_{\cat C}([f],[f']),
	\]
	which is~\eqref{eq:d-as-symmetrized-hom}.

	For the triangle inequality: if either summand on the right is
	$+\infty$ there is nothing to prove.
	Otherwise all three distances are finite and we use enriched
	composition $[f,f']+[f',f'']\le[f,f'']$ and
	$[f'',f']+[f',f]\le[f'',f]$.
	Negating and adding gives
	$d_{\cat C}([f],[f''])\le
	d_{\cat C}([f],[f'])+d_{\cat C}([f'],[f''])$.
\end{proof}

\subsection{Projective nuclei and isometries}

We now impose the mild hypothesis needed to pass $M^*$ and $M_*$ to projective spaces.

\begin{definition}\label{def:nondegenerate}
	We call the profunctor $M$ \emph{nondegenerate} if the Isbell transforms preserve the finite-somewhere condition:
	\[
		\begin{aligned}
			M^*\paren{[\cat C^{\op},\Rbar]_{\mathrm{fs}}} & \subseteq [\cat D,\Rbar]^{\op}_{\mathrm{fs}}, \\
			M_*\paren{[\cat D,\Rbar]^{\op}_{\mathrm{fs}}} & \subseteq [\cat C^{\op},\Rbar]_{\mathrm{fs}}.
		\end{aligned}
	\]
\end{definition}

Under this hypothesis, Lemma~\ref{lem:translation-equivariance} implies that $M^*$ and $M_*$ descend to
well-defined maps
\[
	\begin{aligned}
		M^* & \colon \pcat C\to \pcat D, \\
		M_* & \colon \pcat D\to \pcat C.
	\end{aligned}
\]

\begin{definition}\label{def:projective-nucleus}
	Let $\Nuc(M)$ be the nucleus of $M$ (Definition~\ref{def:nucleus}) and set
	\[
		\Nuc(M)_{\mathrm{fs}}
		:=\Nuc(M)\cap\paren{[\cat C^{\op},\Rbar]_{\mathrm{fs}}\times[\cat D,\Rbar]^{\op}_{\mathrm{fs}}}.
	\]
	The group $\R$ acts on $\Nuc(M)_{\mathrm{fs}}$ by
	\[
		\lambda\cdot(f,g)=(f+\lambda,g-\lambda).
	\]
	The \emph{projective nucleus} is the quotient
	\[
		\pnuc(M):=\Nuc(M)_{\mathrm{fs}}/\R.
	\]
	We metrize $\pnuc(M)$ by
	\[
		d_{\pnuc}\paren{[(f,g)],[(f',g')]}:=\max\set{d_{\cat C}([f],[f']),d_{\cat D}([g],[g'])}.
	\]
\end{definition}

Write $\fix_{\proj}(M_*M^*)\subseteq\pcat C$ and $\fix_{\proj}(M^*M_*)\subseteq\pcat D$
for the images in projective space of the fixed-point sets of the closure operators $M_*M^*$ and $M^*M_*$.

\begin{theorem}[The Isometry Theorem]\label{thm:projective-isometries}
	Let $M\colon \cat C\nrightarrow \cat D$ be a nondegenerate profunctor.
	Then the maps $M^*\colon \pcat C\to \pcat D$ and $M_*\colon \pcat D\to \pcat C$ are $1$-Lipschitz for the metrics
	$d_{\cat C}$ and $d_{\cat D}$.
	Moreover, they restrict to mutually inverse isometries
	\[
		\begin{aligned}
			M^* & \colon \fix_{\proj}(M_*M^*) \xrightarrow{\cong} \fix_{\proj}(M^*M_*), \\
			M_* & \colon \fix_{\proj}(M^*M_*) \xrightarrow{\cong} \fix_{\proj}(M_*M^*).
		\end{aligned}
	\]
	and hence identify $\pnuc(M)$ isometrically with either projective fixed-point set.
\end{theorem}

\begin{proof}
	Functoriality of $M^*$ in the enriched sense gives, for presheaves $f,f'$,
	\[
		[f,f']\le [\cat D,\Rbar]^{\op}(M^*f,M^*f')=[\cat D,\Rbar](M^*f',M^*f),
	\]
	and the same inequality with $f$ and $f'$ exchanged.
	Negating and adding yields
	\[
		-[M^*f,M^*f']-[M^*f',M^*f]\le -[f,f']-[f',f].
	\]
	If $d_{\cat C}([f],[f'])<\infty$, Proposition~\ref{prop:projective-metric} identifies both sides with the corresponding projective
	metrics, giving
	\[
		d_{\cat D}\paren{M^*[f],M^*[f']}\le d_{\cat C}\paren{[f],[f']}.
	\]
	If $d_{\cat C}([f],[f'])=+\infty$, the inequality is tautological.
	Thus $M^*$ is $1$-Lipschitz, and similarly $M_*$.

	On the projective fixed-point sets, $M^*$ and $M_*$ are inverse bijections
	(Proposition~\ref{prop:nucleus-fixedpoints}).
	Since each is $1$-Lipschitz, the two inequalities
	\[
		\begin{aligned}
			d_{\cat D}\paren{M^*[f],M^*[f']} & \le d_{\cat C}\paren{[f],[f']}, \\
			d_{\cat C}\paren{M_*[g],M_*[g']} & \le d_{\cat D}\paren{[g],[g']}.
		\end{aligned}
	\]
	apply to inverse pairs and force equality.
	Hence both restrictions are isometries.

	Finally, the identification with $\pnuc(M)$ is obtained by projecting $[(f,g)]\mapsto [f]$ or $[(f,g)]\mapsto [g]$.
\end{proof}

Consequently we obtain a diagram of metric spaces in which every arrow is an isometry:
\[
	\begin{tikzcd}[column sep=huge,row sep=huge]
		& \pnuc(M)
		\arrow[dl, bend left=15, "i_1"]
		\arrow[dl, bend right=15, "\pi_1"']
		\arrow[dr, bend left=15, "\pi_2"]
		\arrow[dr, bend right=15, "i_2"'] \\
		\fix_{\proj}(M_*M^*)
		\arrow[rr, shift left=0.5ex, "M^*"]
		&&
		\fix_{\proj}(M^*M_*)
		\arrow[ll, shift left=1.0ex, "M_*"]
	\end{tikzcd}
\]
where $\pi_1([(f,g)])=[f]$, $\pi_2([(f,g)])=[g]$, $i_1([f])=[(f,M^*f)]$, and $i_2([g])=[(M_*g,g)]$.

\subsection{External gauge transformations}
Beyond translation by constants there is a larger symmetry: the
additive group $\R^{\cat C}\times\R^{\cat D}$ acts on matrices by
adding constants to rows and columns, the diagonal equivalence of
tropical linear algebra.
This action changes $M$ but preserves $\pnuc(M)$ up to canonical
isometry, so the projective metric and cell decomposition depend
only on the diagonal equivalence class of $M$.

\begin{definition}\label{def:gauge}
	For $u\in\R^{\cat C}$ define $L_u\colon \pcat C\to\pcat C$ by $[f]\mapsto [f-u]$.
	For $v\in\R^{\cat D}$ define $R_v\colon \pcat D\to\pcat D$ by $[g]\mapsto [g-v]$.
	Given $(u,v)$, define the \emph{gauge transform} of $M$ by
	\[
		M^{(u,v)}(c,d):=M(c,d)-u(c)-v(d).
	\]
\end{definition}

\begin{lemma}\label{lem:gauge-conjugacy}
	On projective spaces, the Isbell transforms of $M^{(u,v)}$ are conjugate to those of $M$:
	\[
		\begin{aligned}
			\left(M^{(u,v)}\right)^* & =R_v\circ M^*\circ L_u^{-1}, \\
			\left(M^{(u,v)}\right)_* & =L_u\circ M_*\circ R_v^{-1}.
		\end{aligned}
	\]
\end{lemma}

\begin{proof}
	For a representative $f$ and any $d\in\cat D$,
	\begin{align*}
		\left(M^{(u,v)}\right)^*(f-u)(d)
		 & =\min_{c}\paren{M(c,d)-u(c)-v(d)-(f(c)-u(c))} \\
		 & =\min_{c}\paren{M(c,d)-f(c)}-v(d)             \\
		 & =M^*f(d)-v(d).
	\end{align*}
	This is precisely $(R_v\circ M^*)(f)(d)$.
	The statement for $\left(M^{(u,v)}\right)_*$ is analogous.
\end{proof}

\begin{proposition}\label{prop:gauge-nucleus}
	For any $(u,v)\in\R^{\cat C}\times\R^{\cat D}$, the map
	\[
		\begin{aligned}
			\Phi_{u,v}\colon \pnuc(M) & \to \pnuc\paren{M^{(u,v)}}, \\
			[(f,g)]                   & \mapsto [(f-u,g-v)].
		\end{aligned}
	\]
	is a well-defined isometry with inverse $[(f',g')]\mapsto [(f'+u,g'+v)]$.
	Consequently the projective fixed-point sets and projective nuclei of $M$ and $M^{(u,v)}$ are canonically isometric.
\end{proposition}

\begin{proof}
The maps $L_u$ and $R_v$ are isometries: subtracting the same
potential from both arguments leaves the difference, hence its
oscillation, unchanged.	Lemma~\ref{lem:gauge-conjugacy} identifies the Isbell equations for $M$ with those for $M^{(u,v)}$ under these isometries.
	Thus $(f,g)\in\Nuc(M)$ if and only if $(f-u,g-v)\in\Nuc\paren{M^{(u,v)}}$, and the induced map on projective quotients is an isometry.
\end{proof}

\subsection{Witness cells and the gap matrix}\label{sec:witness_cells}

The projective metric of $\pnuc(M)$ is intrinsic and gauge-invariant
(Proposition~\ref{prop:gauge-nucleus}).
When the indexing sets are finite, the Isbell inequalities also
endow $\pnuc(M)$ with a polyhedral stratification.
The bridge between metric and polyhedral structure is the
\emph{gap matrix} $\delta^{(f,g)}$: its zero entries record the
witness relation and hence the combinatorial cell, while its
positive entries measure slack.

We continue with $\cat C$ and $\cat D$ finite as in
\S\ref{subsec:finite-index}, though the definitions below extend
verbatim to general small $\Rbar$-categories by replacing minima
with infima.
Fix a nondegenerate profunctor $M$.
For a presheaf $f$, write $g:=M^*f$ as in
\eqref{eq:Isbell-min-formulas}, and dually $f:=M_*g$.

\begin{definition}\label{def:witness}
	Let $f\colon \cat C\to\Rbar$ and let $g:=M^*f$.
	An element $c\in\cat C$ is a \emph{witness for $f$ at $d\in\cat D$}
	if $c$ realizes the minimum defining $M^*f(d)$:
	\[
		g(d)=M(c,d)-f(c).
	\]
	Dually, if $g\colon \cat D\to\Rbar$ and $f:=M_*g$, then
	$d\in\cat D$ is a \emph{witness for $g$ at $c\in\cat C$} if $d$
	realizes the minimum defining $M_*g(c)$.
\end{definition}

Because $M^*(f+\lambda)=M^*f-\lambda$ and
$M_*(g-\lambda)=M_*g+\lambda$ for every $\lambda\in\R$, the
witness relation depends only on the projective classes
$[f]\in\pcat C$ and $[g]\in\pcat D$.
By definition of $g=M^*f$, the inequalities
\[
	f(c)+g(d)\le M(c,d),\quad c\in\cat C,\ d\in\cat D
\]
always hold.
The gap matrix measures how far they are from equality.

\begin{definition}\label{def:gap}
	Let $[f]\in\pcat C$ and choose a representative
	$f\colon \cat C\to\Rbar$.
	Set $g:=M^*f$.
	The \emph{gap matrix} of $[f]$ is the function
	\begin{equation}\label{eq:gap}
		\delta^{f}(c,d):=M(c,d)-\paren{f(c)+g(d)}.
	\end{equation}
	Dually, for $[g]\in\pcat D$ with representative $g$ and
	$f:=M_*g$, we set
	$\delta^{g}(c,d):=M(c,d)-\paren{f(c)+g(d)}$.
	If $(f,g)\in\Nuc(M)$, then $\delta^{f}=\delta^{g}$, and we write
	$\delta^{(f,g)}$.
\end{definition}

\begin{lemma}\label{lem:gap-invariances}
	\begin{enumerate}
		\item[(a)] For every $\lambda\in\R$, one has
		      $\delta^{f+\lambda}=\delta^{f}$.
		\item[(b)] For $u\in\R^{\cat C}$ and $v\in\R^{\cat D}$, let
		      $M^{(u,v)}(c,d):=M(c,d)-u(c)-v(d)$ be the gauge
		      transform.
		      If $g=M^*f$, then
		      $\left(M^{(u,v)}\right)^*(f-u)=g-v$, and the
		      corresponding gap matrices agree:
		      \[
			      \delta^{f}=\delta^{f-u},
		      \]
		      where $\delta^{f}$ is computed with $M$ and
		      $\delta^{f-u}$ with $M^{(u,v)}$.
	\end{enumerate}
\end{lemma}

\begin{proof}
	(a) follows from $M^*(f+\lambda)=M^*f-\lambda$ and
	$(f(c)+\lambda)+(g(d)-\lambda)=f(c)+g(d)$.
	For~(b), compute
	\[
		M^{(u,v)}(c,d)-\paren{(f-u)(c)+(g-v)(d)}
		=M(c,d)-\paren{f(c)+g(d)}.
		\qedhere
	\]
\end{proof}

On the finite locus, the zeros of $\delta$ are exactly the witness
pairs.
The next lemma isolates the only subtlety: in $\Rbar$, a zero gap
forces finiteness.

\begin{lemma}\label{lem:gap-zero-finite}
	If $\delta^{f}(c,d)=0$, then $f(c)$, $g(d)$, and $M(c,d)$ are
	all finite real numbers.
\end{lemma}

\begin{proof}
	By definition,
	$\delta^{f}(c,d)=M(c,d)-\paren{f(c)+g(d)}$ is a residuation.
	If $f(c)+g(d)=+\infty$ then
	$\delta^{f}(c,d)=M(c,d)-\infty\in\set{-\infty,\infty}$,
	never~$0$.
	If $f(c)+g(d)=-\infty$ then
	$\delta^{f}(c,d)=M(c,d)-(-\infty)=\infty$, again not~$0$.
	Thus $f(c)+g(d)\in\R$.
	Since $-\infty$ is absorbing for addition, a finite sum forces
	$f(c),g(d)\in\R$.
	Finally, $M(c,d)\in\R$ as well: if $M(c,d)=\pm\infty$ then
	$M(c,d)-(f(c)+g(d))=\pm\infty$.
\end{proof}

\begin{proposition}\label{prop:gap}
	Let $[f]\in\pcat C$ with representative $f$, let $g:=M^*f$, and
	let $\delta=\delta^{f}$.
	Then:
	\begin{enumerate}
		\item[(a)] $\delta(c,d)\ge 0$ for all
		      $(c,d)\in\cat C\times\cat D$.
		\item[(b)] If $\delta(c,d)=0$, then $c$ is a witness for $f$
		      at~$d$.
		\item[(c)] If $f$ and $g$ are finite-valued, then
		      $\delta(c,d)=0$ if and only if $c$ is a witness for $f$
		      at~$d$.
		      In particular, every column contains at least one zero.
		\item[(d)] If moreover $(f,g)\in\Nuc(M)$ and $f,g$ are
		      finite-valued, then $\delta(c,d)=0$ if and only if $d$
		      is a witness for $g$ at~$c$.
		      In particular, every row contains at least one zero.
	\end{enumerate}
\end{proposition}

\begin{proof}
	(a) Since $g(d)=\min_{c'}\paren{M(c',d)-f(c')}$, we have
	$g(d)\le M(c,d)-f(c)$ for every $c$.
	By residuation this is equivalent to $f(c)+g(d)\le M(c,d)$,
	hence $\delta(c,d)\ge 0$.

	(b) If $\delta(c,d)=0$, then
	Lemma~\ref{lem:gap-zero-finite} shows the relevant entries are
	finite, so subtraction is ordinary:
	$0=M(c,d)-f(c)-g(d)$, hence $g(d)=M(c,d)-f(c)$ and $c$
	realizes the minimum in $M^*f(d)$.

	(c) If $f,g$ are finite-valued and $c$ is a witness for $f$
	at~$d$, then $g(d)=M(c,d)-f(c)$ and $\delta(c,d)=0$.
	The converse is~(b).
	Since $\cat C$ is finite, every minimum defining $g(d)$ is
	attained, so every column contains a witness and hence a zero.

	(d) Apply~(c) to the dual description $f=M_*g$.
\end{proof}

\begin{corollary}\label{cor:pairs}
	If $(f,g)\in\pnuc(M)$ and $f,g$ are finite-valued, then $c$ is a
	witness for $f$ at~$d$ if and only if $d$ is a witness for $g$
	at~$c$.
\end{corollary}

\begin{proof}
	This is Proposition~\ref{prop:gap}\textup{(c)}
	and~\textup{(d)}.
\end{proof}

\begin{corollary}\label{cor:witness-fixedpoints}
	A finite-valued presheaf $f$ is a fixed point of $M_*M^*$ if and
	only if every row of $\delta^{f}$ contains a zero.
	Equivalently, for every $c\in\cat C$ there exists $d\in\cat D$
	such that $d$ is a witness for $M^*f$ at~$c$.
\end{corollary}

\begin{proof}
	If $f$ is finite-valued and fixed by $M_*M^*$, write $g:=M^*f$,
	so $f=M_*g$.
	Since $\cat D$ is finite, for each $c$ the minimum defining
	$f(c)$ is attained at some $d$, and then $\delta^{f}(c,d)=0$.
	Conversely, if every row contains a zero, fix $c$ and choose $d$
	with $\delta^{f}(c,d)=0$.
	Then $f(c)=M(c,d)-g(d)$, hence
	\[
		(M_*g)(c)=\min_{d'}\paren{M(c,d')-g(d')}\le M(c,d)-g(d)=f(c).
	\]
	On the other hand $f\le M_*M^*f=M_*g$ by~\eqref{eq:closure}, so
	equality holds coordinatewise.
\end{proof}

For a finite-valued nucleus point, the witness pattern is exactly
the zero set of the gap matrix.

\begin{definition}\label{def:Z}
	For $(f,g)\in\pnuc(M)$ with $f,g$ finite-valued, define the
	\emph{witness relation}
	\begin{equation}\label{eq:Z}
		Z(f,g):=\set{(c,d)\in\cat C\times\cat D\mid
		\delta^{(f,g)}(c,d)=0}.
	\end{equation}
\end{definition}

By Proposition~\ref{prop:gap}\textup{(d)}, the relation $Z(f,g)$
meets every row and every column.
These relations partition the finite part of $\pnuc(M)$: we declare
$(f,g)\sim(f',g')$ if $Z(f,g)=Z(f',g')$.

\begin{definition}\label{def:witness-cell}
	For a relation $Z\subseteq \cat C\times\cat D$ meeting every row
	and column, define the \emph{open witness cell}
	\[
		\Cell^{\circ}(Z):=\set{(f,g)\in\pnuc(M)\mid f,g\text{
		finite-valued and }Z(f,g)=Z}.
	\]
\end{definition}

\begin{remark}\label{rem:cell-polytope}
	After choosing an affine chart for the projective quotient (for
	example $\min f=0$), the closure of $\Cell^{\circ}(Z)$ is a
	classical polytope, cut out by the equalities
	$M(c,d)=f(c)+g(d)$ for $(c,d)\in Z$ together with the
	inequalities $M(c,d)\ge f(c)+g(d)$ for all $(c,d)$.
	The open cell $\Cell^{\circ}(Z)$ is its relative interior.
	We develop this in Section~\ref{subsec:polyhedral_pnuc}.
\end{remark}
Because $\cat D$ is discrete, the copresheaf represented by
$d\in\cat D$ is the delta function $g_d(d)=0$ and
$g_d(d')=-\infty$ for $d'\neq d$.
A direct computation gives $M_*g_d=M(-,d)$: the columns of the
matrix $M$ are the images of the representables.

\begin{definition}\label{def:anchors}
	For $d\in\cat D$ define the $d$th \emph{anchor presheaf}
	$A_d\colon\cat{C}\to \Rbar$ by $A_d=M(-,d)$.
\end{definition}

\begin{proposition}\label{prop:anchors}
	Each anchor $A_d$ lies in
	$\fix(M_*M^*)=\im(M_*)$.
	Moreover, for any $f\in\fix(M_*M^*)$ there exist weights
	$\lambda_d\in\Rbar$ such that
	\begin{equation}\label{eq:anchor-hull}
		f(c)=\min_{d\in\cat D}\paren{A_d(c)-\lambda_d},\quad
		c\in\cat C.
	\end{equation}
\end{proposition}

\begin{proof}
	Since $A_d=M_*g_d$, we have
	$A_d\in\im(M_*)=\fix(M_*M^*)$.
	Conversely, if $f\in\fix(M_*M^*)=\im(M_*)$ then $f=M_*g$ for
	some copresheaf $g$, hence
	\[
		f(c)=\min_{d\in\cat D}\paren{M(c,d)-g(d)}
		=\min_{d\in\cat D}\paren{A_d(c)-\lambda_d}
	\]
	with $\lambda_d:=g(d)$.
\end{proof}

Equivalently, $\fix(M_*M^*)$ is the closure of the anchors $A_d$
under weighted coproducts: every fixed point $f$ can be written as
the pointwise minimum of translates $A_d-\lambda_d$, and conversely
every such minimum lies in $\fix(M_*M^*)$.

\begin{proposition}\label{prop:compact}
	If $M\colon\cat C\times\cat D\to\R$ is a real-valued matrix on
	finite sets, then $\pnuc(M)$ is compact.
\end{proposition}

\begin{proof}
	Since $M$ is real-valued, each anchor $A_d=M(-,d)$ is
	real-valued, and Proposition~\ref{prop:anchors} then implies
	that every fixed point of $M_*M^*$ admits a real-valued
	representative.
	On the affine slice $\set{f\mid f(c_0)=0}$, the Isbell
	inequalities $f(c)+g(d)\le M(c,d)$ bound $f$ and $g$
	coordinatewise, so the nucleus embeds as a closed bounded
	subset of $\R^{|\cat C|-1}$.
\end{proof}

\section{Polyhedral geometry, event loci, and lattice towers}\label{subsec:polyhedral_pnuc}

In Section~\ref{sec:witness_cells} we associated to each finite-valued nucleus point $(f,g)$ its \emph{gap matrix}
\[
	\delta^{(f,g)}(c,d):=M(c,d)-f(c)-g(d)
\]
and its zero set
\[
	Z(f,g):=\set{(c,d)\in\cat C\times\cat D\mid \delta^{(f,g)}(c,d)=0}.
\]
This section develops the polyhedral side of that data.
We pass from the witness relation to explicit witness polyhedra, prove that positive gap values are exact distances to event loci, refine the witness decomposition by order chambers, and show that pointed thresholding of the gap matrix yields a chamberwise constructible sheaf of concept-lattice towers.

\subsection{Witness polyhedra and admissibility}

We continue with $\cat C$ and $\cat D$ finite and now assume that
$M\colon\cat C\times\cat D\to\R$ has finite real entries.
By Proposition~\ref{prop:compact}, $\pnuc(M)$ is compact and each
projective class admits a real-valued representative.

To describe the polyhedral structure concretely, we fix an affine
chart.
Choose $c_0\in\cat C$ and set
\[
	\Nuc(M)_0:=\set{(f,g)\in\Nuc(M)\mid f(c_0)=0}.
\]
Since each class in $\pnuc(M)$ has a unique representative in
$\Nuc(M)_0$, this identifies $\Nuc(M)_0\cong\pnuc(M)$ as sets;
we use this chart when explicit coordinates are needed, but the
metric and cell structures are independent of the choice of~$c_0$.

Working in this chart, the Isbell inequalities
\[
	f(c)+g(d)\le M(c,d)\text{ for all }c\in\cat C,d\in\cat D
\]
cut out a feasibility polyhedron in $\R^{\cat C}\times\R^{\cat D}$.
The Isbell fixed-point conditions for the nucleus are equivalent to the requirement that every row and every column attains
equality.
Given $Y\subseteq\cat C\times\cat D$, we obtain a smaller closed polyhedron by forcing equality along $Y$.

\begin{definition}\label{def:cover}
	Let $Y\subseteq\cat C\times\cat D$.
	We say that $Y$ \emph{covers $\cat C$} if for every $c\in\cat C$ there exists $d\in\cat D$ with $(c,d)\in Y$.
	We say that $Y$ \emph{covers $\cat D$} if for every $d\in\cat D$ there exists $c\in\cat C$ with $(c,d)\in Y$.
\end{definition}

\begin{definition}\label{def:CellY}
	For $Y\subseteq\cat C\times\cat D$, define $\Cell(Y)$ to be the set of pairs of real-valued functions
	$(f,g)\in\R^{\cat C}\times\R^{\cat D}$ satisfying the gauge condition $f(c_0)=0$ and the constraints
	\begin{align}
		f(c)+g(d) & \le M(c,d)\text{ for all }(c,d)\in\cat C\times\cat D,\label{eq:Cell-ineq} \\
		f(c)+g(d) & = M(c,d)\text{ for all }(c,d)\in Y.\label{eq:Cell-eq}
	\end{align}
\end{definition}

Thus $\Cell(Y)$ is a (possibly empty) closed polyhedron.
For general $Y$, feasibility of \eqref{eq:Cell-ineq}--\eqref{eq:Cell-eq} does not imply that $(f,g)$ is a nucleus point.
The next lemma isolates the combinatorial condition that forces the Isbell equalities in every row or column.

\begin{lemma}\label{lem:Cell-implies-Isbell}
	Let $Y\subseteq\cat C\times\cat D$ and let $(f,g)$ satisfy \eqref{eq:Cell-ineq}--\eqref{eq:Cell-eq}.
	\begin{enumerate}
		\item[(a)] If $Y$ covers $\cat D$, then $g=M^*f$.
		\item[(b)] If $Y$ covers $\cat C$, then $f=M_*g$.
	\end{enumerate}
	Consequently, if $Y$ covers both $\cat C$ and $\cat D$, then every $(f,g)\in\Cell(Y)$ lies in $\Nuc(M)_0$.
\end{lemma}

\begin{proof}
	(a) Fix $d\in\cat D$.
	From \eqref{eq:Cell-ineq} we obtain $g(d)\le M(c,d)-f(c)$ for all $c$, hence
	\[
		g(d)\le \min_{c\in\cat C}\paren{M(c,d)-f(c)}=(M^*f)(d).
	\]
	Since $Y$ covers $\cat D$, there exists $c$ with $(c,d)\in Y$.
	Then \eqref{eq:Cell-eq} gives $g(d)=M(c,d)-f(c)$, hence $g(d)\ge (M^*f)(d)$.
	Thus $g(d)=(M^*f)(d)$.

	(b) The proof is symmetric.
	If $Y$ covers both sides, (a) and (b) give $g=M^*f$ and $f=M_*g$, so $(f,g)\in\Nuc(M)$, and the gauge condition places it
	in $\Nuc(M)_0$.
\end{proof}

Covering both $\cat C$ and $\cat D$ is necessary for $\Cell(Y)$
to lie in the nucleus (Lemma~\ref{lem:Cell-implies-Isbell}), but
for most covering sets the system
\eqref{eq:Cell-ineq}--\eqref{eq:Cell-eq} is infeasible.

\begin{definition}\label{def:admissible}
	A subset $Y\subseteq\cat C\times\cat D$ is \emph{admissible} if it covers both $\cat C$ and $\cat D$ and $\Cell(Y)\neq\varnothing$.
\end{definition}

When $Y$ is admissible, $\Cell(Y)$ is a nonempty polyhedron consisting entirely of nucleus points by
Lemma~\ref{lem:Cell-implies-Isbell}.

\begin{lemma}\label{lem:Z-gives-Cell}
	Let $(f,g)\in\Nuc(M)_0$, and set
	\[
		Z(f,g)=\set{(c,d)\in\cat C\times\cat D\mid f(c)+g(d)=M(c,d)}.
	\]
	Then $(f,g)\in\Cell(Z(f,g))$. In particular, $Z(f,g)$ is admissible.
\end{lemma}

\begin{proof}
	The inequalities \eqref{eq:Cell-ineq} are exactly the Isbell inequalities for $(f,g)$.
	The definition of $Z(f,g)$ is the equality condition \eqref{eq:Cell-eq}.
	Since $\cat C$ and $\cat D$ are finite and $(f,g)$ is a nucleus point, each row and each column attains equality, so $Z(f,g)$
	covers both $\cat C$ and $\cat D$.
\end{proof}

\begin{corollary}\label{cor:finite_union_cells}
	The set $\Nuc(M)_0\cong\pnuc(M)$ is a union of finitely many polytopes:
	\[
		\Nuc(M)_0=\bigcup_{Y\ \mathrm{admissible}}\Cell(Y).
	\]
\end{corollary}
\begin{proof}
	By Lemma~\ref{lem:Z-gives-Cell}, each $(f,g)\in\Nuc(M)_0$ lies
	in $\Cell(Z(f,g))$, and $Z(f,g)$ is admissible.
	Since $\cat C\times\cat D$ is finite, there are only finitely
	many admissible subsets $Y$.

	Each $\Cell(Y)$ is a closed convex polyhedron by construction.
	Since $M$ is real-valued,
	Proposition~\ref{prop:compact} implies that $\Nuc(M)_0$ is
	compact, hence each nonempty
	$\Cell(Y)\subseteq\Nuc(M)_0$ is bounded.
	Thus every admissible $\Cell(Y)$ is a polytope.
\end{proof}

\begin{lemma}\label{lem:Cell-intersection}
	For any $Y,Y'\subseteq\cat C\times\cat D$ one has
	\[
		\Cell(Y)\cap\Cell(Y')=\Cell(Y\cup Y').
	\]
\end{lemma}

\begin{proof}
	The inequalities \eqref{eq:Cell-ineq} are common to all $\Cell(\cdot)$, and the equalities \eqref{eq:Cell-eq} imposed by $Y$ and by $Y'$
	together are exactly those imposed by $Y\cup Y'$.
\end{proof}

\begin{corollary}\label{cor:polytopal_complex}
	The admissible polytopes $\set{\Cell(Y)}$ form a finite polytopal complex inside $\Nuc(M)_0\cong\pnuc(M)$:
	if $Y$ and $Y'$ are admissible, then $\Cell(Y)\cap\Cell(Y')$ is either empty or a common face of both.
\end{corollary}

\begin{proof}
	By Lemma~\ref{lem:Cell-intersection}, the intersection is $\Cell(Y\cup Y')$.
	If it is nonempty, then $Y\cup Y'$ covers both $\cat C$ and $\cat D$, hence is admissible, and $\Cell(Y\cup Y')$ is obtained from
	$\Cell(Y)$ and $\Cell(Y')$ by imposing additional linear equalities.
	Therefore it is a face of each.
\end{proof}

\subsection{From witness polyhedra to witness cells}
The witness polyhedra $\Cell(Y)$ are best regarded as closures of the open witness cells from \S\ref{sec:witness_cells}.
For $(f,g)\in\Nuc(M)_0$, write $\delta=\delta^{(f,g)}$ and $Z=Z(f,g)=\delta^{-1}(0)$.
Then $\Cell(Z)$ is the smallest witness polyhedron containing $(f,g)$, and its relative interior consists of those points
for which no additional inequalities become equalities:
\[
	\Cell^{\circ}(Z)
	=\set{(f',g')\in\Cell(Z)\mid \delta^{(f',g')}(c,d)>0\text{ for all }(c,d)\notin Z}.
\]
In particular, for every admissible $Z$ the witness cell $\Cell^{\circ}(Z)$ is the relative interior of the polytope $\Cell(Z)$.

In the classical min-plus setting this recovers the type decomposition of a tropical polytope, for instance via tropical hyperplane
arrangements \cite{develinSturmfels2004tropical}.
The gap matrix controls not only the combinatorics but also the
metric geometry of this decomposition: each positive entry
$\delta^{(f,g)}(c,d)$ is the exact projective distance from
$(f,g)$ to the locus where $(c,d)$ becomes a witness pair.
This is the Events Theorem, proved in
\S\ref{subsec:events} below.

\subsection{The events theorem}\label{subsec:events}

We now show that each positive entry of the gap matrix is the exact
projective distance to the locus where the corresponding inequality
becomes tight.
Related formulas for the distance to a half-space in Hilbert's
projective metric appear in idempotent semimodule
theory~\cite{Nitica01062007}; the result here is stronger because
nucleus points are constrained by the Isbell equations, which
couple all rows and columns simultaneously.

Fix a point $(f,g)\in\pnuc(M)$ and let $\delta=\delta^{(f,g)}$ be its gap matrix.

\begin{definition}
	For each pair $(c,d)\in \cat C \times \cat D$, we define the corresponding \emph{event locus}
	\[
		\mathcal{E}_{c,d}:=\set{(f',g')\in\pnuc(M)\mid \delta^{(f',g')}(c,d)=0}.
	\]
\end{definition}
Thus $\mathcal{E}_{c,d}$ is the locus in $\pnuc(M)$ where $(c,d)$ is a witness pair.

\begin{lemma}\label{lem:events}
	Let $(f,g)\in\pnuc(M)$ and let $\lambda=\delta^{(f,g)}(c_i,d_j)>0$.
	Then there exists $(f',g')\in \mathcal{E}_{c_i,d_j}$ such that
	\[
		d_{\pnuc}\paren{(f,g),(f',g')}=\lambda.
	\]
\end{lemma}

\begin{proof}
	Work in the gauge slice $\Nuc(M)_0$ and choose representatives with $f(c_0)=0$.
	Define $f''\colon\cat C\to\R$ by
	\[
		f''(c)=
		\begin{cases}
			f(c)           & c\ne c_i, \\
			f(c_i)+\lambda & c=c_i.
		\end{cases}
	\]
	Set $g':=M^*f''$ and $f':=M_*g'=M_*M^*f''$.
	By the identities $M^*M_*M^*=M^*$ and $M_*M^*M_*=M_*$, we have $g'=M^*f'$ and $f'=M_*g'$, hence $(f',g')\in\Nuc(M)_0$.

	For each $d\in\cat D$ one has
	\begin{align*}
		g'(d)
		 & = \min\left(\min_{c\ne c_i}\paren{M(c,d)-f(c)},\, M(c_i,d)-f(c_i)-\lambda\right) \\
		 & = \min\left(g(d),\, g(d)+\delta^{(f,g)}(c_i,d)-\lambda\right)                    \\
		 & = g(d)-\max\paren{\lambda-\delta^{(f,g)}(c_i,d),0}.
	\end{align*}
	In particular,
	\[
		g(d)-g'(d)=\max\paren{\lambda-\delta^{(f,g)}(c_i,d),0}
	\]
	takes values in $[0,\lambda]$.
	Moreover, $g(d_j)=g'(d_j)$ because $\delta^{(f,g)}(c_i,d_j)=\lambda$.
	Since $(f,g)\in\Nuc(M)_0$, row $c_i$ of $\delta^{(f,g)}$ contains a zero by Proposition~\ref{prop:gap}\textup{(d)}; choose $d_k$ with $\delta^{(f,g)}(c_i,d_k)=0$.
	Then $g(d_k)-g'(d_k)=\lambda$.
	Therefore $d_{\cat D}([g],[g'])=\lambda$.

	Since $(f,g)$ and $(f',g')$ are nucleus points, $f=M_*g$ and $f'=M_*g'$.
	The map $M_*$ is $1$-Lipschitz for the projective metrics by Theorem~\ref{thm:projective-isometries}, so
	\[
		d_{\cat C}([f],[f'])\le d_{\cat D}([g],[g'])=\lambda.
	\]
	By definition of $d_{\pnuc}$ it follows that
	\[
		d_{\pnuc}\paren{(f,g),(f',g')}=\lambda.
	\]

	Finally, at the distinguished pair $(c_i,d_j)$ we have $g'(d_j)=g(d_j)$ and
	\[
		f'(c_i)
		=\min_{d\in\cat D}\paren{M(c_i,d)-g'(d)}
		\le M(c_i,d_j)-g'(d_j)
		=M(c_i,d_j)-g(d_j)
		=f(c_i)+\lambda.
	\]
	On the other hand, $f''\le f'=M_*M^*f''$ by extensivity of the closure operator $M_*M^*$, so
	$f'(c_i)\ge f''(c_i)=f(c_i)+\lambda$.
	Thus $f'(c_i)=f(c_i)+\lambda$, and
	\[
		\delta^{(f',g')}(c_i,d_j)
		=M(c_i,d_j)-f'(c_i)-g'(d_j)
		=\delta^{(f,g)}(c_i,d_j)-\lambda
		=0.
	\]
	Hence $(f',g')\in\mathcal{E}_{c_i,d_j}$, as claimed.
\end{proof}

The construction in the proof has a geometric interpretation worth
highlighting.
The perturbation $f''$ moves $f$ by $\lambda$ in the
$c_i$-coordinate direction, which generically leaves the nucleus.
The closure $f' = M_*M^*f''$ projects back onto the nucleus, and
the resulting point lies on the event locus $\mathcal{E}_{c_i,d_j}$
at projective distance exactly $\lambda$ from $(f,g)$.
Expressing $f'$ directly in terms of the gap matrix:
\[
	f'(c)
	=f(c)+\min_{d\in\cat D}\left(\delta(c,d)
	+\max\paren{\lambda-\delta(c_i,d),0}\right).
\]
This formula, together with the analogous expression for $g'$,
determines the figures in the examples below.  

Note that the closure $M_*M^*$ is nonexpansive, so it can only
decrease the distance from a nucleus point:
$d([f],[M_*M^*f''])\le d([f],[f''])$ for any presheaf $f''$.
In general this inequality is strict.
The key feature of the construction is that the perturbation
$f''=f+\lambda\mathbf{1}_{c_i}$ is calibrated so that no distance
is lost: the extensivity $f''\le M_*M^*f''$ pins $f'(c_i)$ at
$f(c_i)+\lambda$ from below, while the Isbell equation pins it
from above, forcing $d([f],[f'])=\lambda$.

Lemma~\ref{lem:events} gives the upper bound: there exists a point
on $\mathcal{E}_{c,d}$ at distance exactly $\lambda$.
The theorem shows this is optimal.
For any subset $S\subseteq\pnuc(M)$, write
$d_{\pnuc}((f,g),S)
=\inf_{(f',g')\in S}d_{\pnuc}((f,g),(f',g'))$.

\begin{theorem}[The Events Theorem]\label{thm:events}
	Let $(f,g)\in\pnuc(M)$ and let $(c,d)\in\cat C\times\cat D$.
	Then
	\[
		d_{\pnuc}\paren{(f,g),\mathcal{E}_{c,d}}=\delta^{(f,g)}(c,d).
	\]
\end{theorem}

\begin{proof}
	If $\delta^{(f,g)}(c,d)=0$, then $(f,g)\in\mathcal{E}_{c,d}$ and both sides are $0$.
	If $M(c,d)=+\infty$, then $\delta^{(f,g)}(c,d)=+\infty$ at every nucleus point
	and $\mathcal{E}_{c,d}=\varnothing$, so both sides are $+\infty$.
	Assume $\lambda=\delta^{(f,g)}(c,d)\in(0,\infty)$.
	By Lemma~\ref{lem:events} there exists $(f',g')\in\mathcal{E}_{c,d}$ with
	$d_{\pnuc}\paren{(f,g),(f',g')}=\lambda$, so
	$d_{\pnuc}\paren{(f,g),\mathcal{E}_{c,d}}\le\lambda$.

	For the reverse inequality, work in the gauge slice $\Nuc(M)_0$ and fix $(f_1,g_1)\in\mathcal{E}_{c,d}$.
	Set $a(c)=f_1(c)-f(c)$ and let
	\[
		\alpha=\min_{c\in\cat C}a(c),\qquad \beta=\max_{c\in\cat C}a(c).
	\]
	By definition of the projective metric on $\pcat C$ one has
	\[
		d_{\cat C}\paren{[f],[f_1]}=\beta-\alpha.
	\]
	Since $g=M^*f$ and $g_1=M^*f_1$, for each $d'\in\cat D$,
	\[
		g_1(d')
		=\min_{c\in\cat C}\paren{M(c,d')-f_1(c)}
		=\min_{c\in\cat C}\paren{M(c,d')-f(c)-a(c)}.
	\]
	Therefore
	\[
		g(d')-\beta\le g_1(d')\le g(d')-\alpha,
	\]
	so $g_1(d')-g(d')\in[-\beta,-\alpha]$.
	It follows that for all $c'\in\cat C$ and $d'\in\cat D$,
	\[
		\paren{f_1(c')+g_1(d')}-\paren{f(c')+g(d')}
		=a(c')+\paren{g_1(d')-g(d')}
		\in[\alpha-\beta,\beta-\alpha],
	\]
	hence
	\[
		\left|\paren{f_1(c')+g_1(d')}-\paren{f(c')+g(d')}\right|\le \beta-\alpha.
	\]
	Evaluating at the distinguished pair $(c,d)$ and using $f_1(c)+g_1(d)=M(c,d)$ gives
	\[
		\lambda
		=M(c,d)-f(c)-g(d)
		=\paren{f_1(c)+g_1(d)}-\paren{f(c)+g(d)}
		\le \beta-\alpha.
	\]
	Finally, by definition of $d_{\pnuc}$ one has
	\[
		d_{\cat C}\paren{[f],[f_1]}\le d_{\pnuc}\paren{(f,g),(f_1,g_1)},
	\]
	so $\lambda\le d_{\pnuc}\paren{(f,g),(f_1,g_1)}$ for every $(f_1,g_1)\in\mathcal{E}_{c,d}$.
	Taking the infimum over $\mathcal{E}_{c,d}$ yields
	$d_{\pnuc}\paren{(f,g),\mathcal{E}_{c,d}}\ge\lambda$, hence equality.
\end{proof}

\begin{corollary}\label{cor:first-event-radius}
	Let $(f,g)\in\pnuc(M)$ and write $Z=Z(f,g)$.
	Then
	\[
		d_{\pnuc}\paren{(f,g),\bigcup_{(c,d)\notin Z}\mathcal E_{c,d}}
		=
		\min\set{\delta^{(f,g)}(c,d)\mid (c,d)\notin Z}.
	\]
	In particular, the smallest positive gap is the first radius at which an additional witness can appear.
\end{corollary}

\begin{proof}
	Because $\cat C\times\cat D$ is finite, the distance from $(f,g)$ to the union of the event loci is the minimum of the distances to the individual event loci.
	Apply Theorem~\ref{thm:events} to each pair $(c,d)\notin Z$.
\end{proof}

We illustrate the perturb-and-project construction on the
$3\times 4$ example from \S\ref{sec:witness_cells}.
Let $f=(0,0,0)$ and $g=M^*f=(0.7,-1.6,0.1,-2.9)$, viewed in the gauge slice $\Nuc(M)_0$.
The gap matrix is
\[
	\delta^{(f,g)}=
	\begin{bmatrix}
		0   & 3.1 & 1.6 & 1.6 \\
		0.5 & 4.2 & 0   & 5.1 \\
		1.3 & 0   & 1.9 & 0
	\end{bmatrix}.
\]

Figure~\ref{fig:cellular_event_thm} isolates the event with
\[
	\lambda=1.9=\delta^{(f,g)}(c_2,d_3).
\]
The shaded hexagon is the radius-$\lambda$ projective ball about $f$ in $\pcat C\cong\R^2$.
Define $f''\colon\cat C\to\R$ by $f''(c)=f(c)$ for $c\ne c_2$ and $f''(c_2)=f(c_2)+\lambda$.
This presheaf need not be a fixed point of $M_*M^*$, and in fact
is not: the red point $f''$ lies on the boundary of the ball but
outside the nucleus.
Applying the closure produces $f'=M_*M^*f''$, shown in blue, which lies on the event locus $\mathcal E_{c_2,d_3}$ and remains at projective distance $\lambda$ from $f$.
In this example one finds $f'=(0.6,0,1.9)$, which is projectively equivalent to $(0,-0.6,1.3)$ in the gauge $f(c_0)=0$.

\begin{figure}[t]
	\centering
	\includegraphics[width=0.74\linewidth]{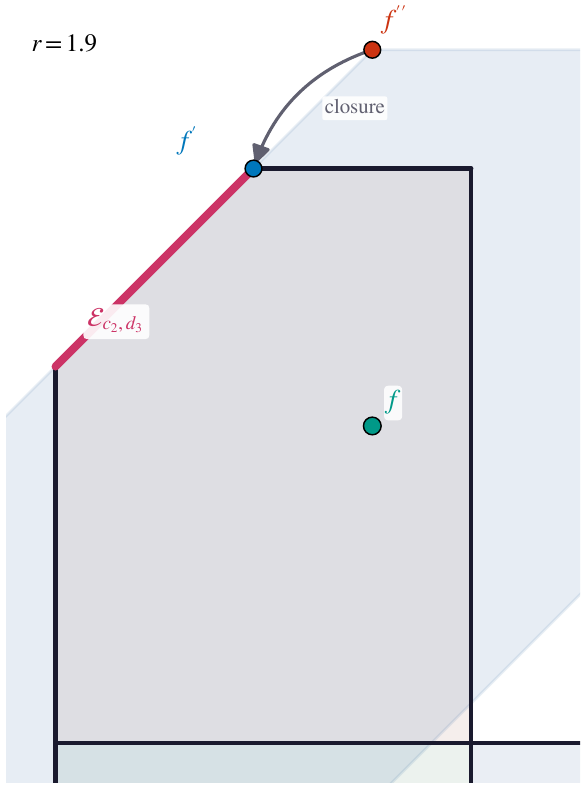}
	\caption{The event with radius $1.9=\delta^{(f,g)}(c_2,d_3)$.
		The red point $f''$ is obtained by moving a distance $1.9$ in the $c_2$-direction; it lies on the radius-$1.9$ projective ball but outside the nucleus.
		Its closure $f'=M_*M^*f''$ lies on the event locus $\mathcal E_{c_2,d_3}$, showing that this locus is first met at radius $1.9$.}
	\label{fig:cellular_event_thm}
\end{figure}
\subsection{Order chambers}\label{subsec:order_chambers}

A witness cell is determined by the zero pattern
$Z=\delta^{-1}(0)$ of the gap matrix.
Theorem~\ref{thm:events} shows that the remaining entries carry
metric information: for each $(c,d)\notin Z$ the gap value
$\delta(c,d)$ is the distance from $(f,g)$ to the event locus
$\mathcal{E}_{c,d}$.
Keeping track only of the relative order of the positive gaps
refines the witness decomposition.

\begin{definition}\label{def:order-chamber}
	Fix $(f,g)\in\pnuc(M)$ and write $\delta=\delta^{(f,g)}$.
	The \emph{gap preorder} $\preceq_{f,g}$ on
	$\cat C\times\cat D$ is the total preorder defined by
	\[
		(c,d)\preceq_{f,g}(c',d')
		\iff \delta(c,d)\le \delta(c',d').
	\]
	Two nucleus points in the same witness cell lie in the same
	\emph{order chamber} if they induce the same gap preorder.
\end{definition}

To describe a chamber as a polyhedron, let $\preceq$ be a total
preorder on $\cat C\times\cat D$ and let
$Y\subseteq\cat C\times\cat D$ be its set of minimal elements.
Assume that $Y$ covers $\cat C$ and $\cat D$, so that
$\Cell(Y)\subseteq\Nuc(M)_0$.
The closure of the corresponding order chamber is the subset of
$\Cell(Y)$ cut out by the weak inequalities
\[
	(c,d)\preceq (c',d')
	\Rightarrow\delta^{(f,g)}(c,d)\le \delta^{(f,g)}(c',d').
\]
The order chamber itself is the relative interior, obtained by
requiring strict inequality between distinct equivalence classes.
These conditions are gauge-invariant by
Lemma~\ref{lem:gap-invariances}.

\begin{proposition}\label{prop:order-chambers}
	The order chambers subdivide each witness cell into a finite
	polyhedral complex, cut out by a hyperplane arrangement in the
	gap values.
	A codimension-one wall corresponds to a tie
	$\delta(c,d)=\delta(c',d')$ between two gap values from
	distinct equivalence classes of $\preceq$; crossing the wall
	swaps their order and leaves the rest of the preorder unchanged.
	In particular, the adjacency graph of order chambers is
	bipartite~\cite{aguiarMahajan2017hyperplaneArrangements}.
\end{proposition}

\begin{figure}[t]
	\centering
	\includegraphics[width=0.90\linewidth]{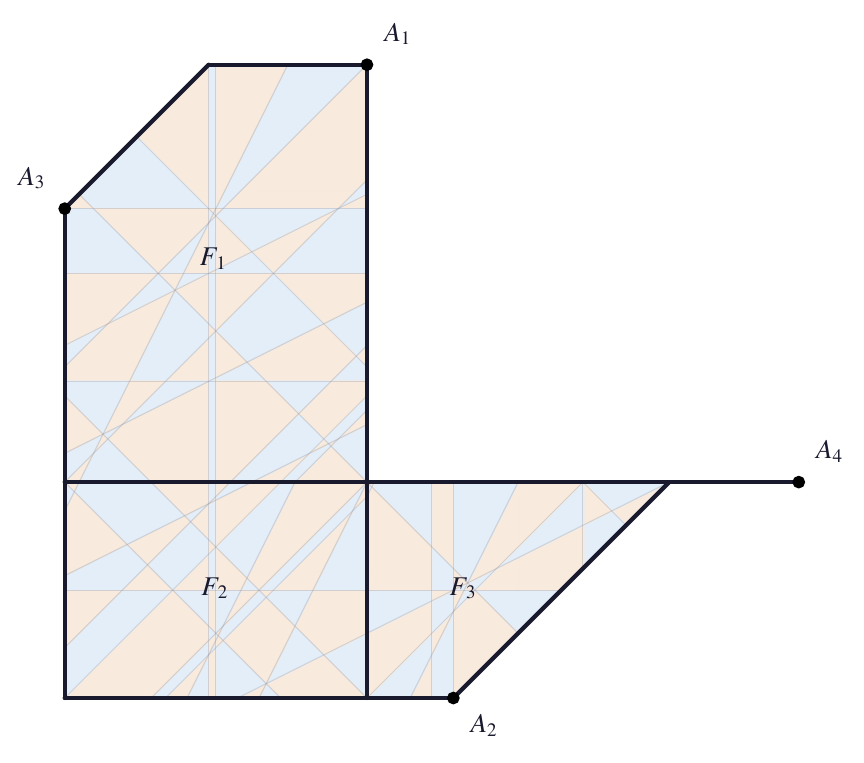}
	\caption{The global order-chamber refinement of the
		witness-cell decomposition in the running example.
		Each witness cell is subdivided by comparing the positive
		gap values $\delta(c,d)$, and adjacent chambers differ by
		swapping the order of one neighboring pair of event radii.
		The two colors indicate the bipartite coloring of the
		chamber graph.}
	\label{fig:order_chambers}
\end{figure}

\subsection{Pointed thresholding and formal concept
lattices}\label{subsec:thresholding}

The nucleus of a Boolean relation is a lattice of formal concepts
in the sense of Wille~\cite{Wille1982}; the nucleus of a
real-valued matrix is the polyhedral metric space studied in the
preceding sections.
Thresholding the gap matrix connects the two: at each nucleus point
$(f,g)$ and each radius $\varepsilon>0$, the sublevel set
$\{\delta^{(f,g)}(c,d)\le\varepsilon\}$ is a Boolean relation whose
concept lattice $L_\varepsilon$ captures the discrete
Galois-closed dependencies among the incidences visible from
$(f,g)$ at resolution~$\varepsilon$.
As $\varepsilon$ increases, the relation grows and the lattice with it,
assembling a tower of concept lattices that records the order in
which the entries of $M$ are reconstructed from the witness pairs
outward.

To describe the lattice $L_\varepsilon$ explicitly, we briefly
review the Boolean case.
By Booleans, we mean the two-element monoidal poset $\{0,1\}$ with
order $0\leq 1$ and monoidal product~$\wedge$.
A Boolean-valued profunctor on discrete sets $\cat C,\cat D$
amounts to a relation $R\subseteq\cat C\times\cat D$.
Pre- and copresheaves are Boolean-valued functions and may be
identified with subsets.
The Isbell conjugates of $R$ are the pair of order-reversing maps
between power sets
\[
	\begin{aligned}
		R^*\colon \cat P(\cat C) & \to \cat P(\cat D),                                       &
		R^*(F)                   & =\set{d\in\cat D\mid (c,d)\in R \text{ for all } c\in F},   \\
		R_*\colon \cat P(\cat D) & \to \cat P(\cat C),                                       &
		R_*(G)                   & =\set{c\in\cat C\mid (c,d)\in R \text{ for all } d\in G},
	\end{aligned}
\]
and they form a Galois connection: $F\subseteq R_*(G)\iff G\subseteq R^*(F)$.
A \emph{formal concept} of $R$ is a fixed point of this
adjunction---a pair $(F,G)$ with $G=R^*(F)$ and
$F=R_*(G)$---exactly as in Definition~\ref{def:nucleus}, but
enriched over Booleans rather than~$\Rbar$.
Following~\cite{Wille1982}, $F$ is called the \emph{extent} and
$G$ the \emph{intent}.
The set of all formal concepts is a complete lattice, ordered by
inclusion of extents (equivalently, reverse inclusion of intents):
\[(F,G)\leq (F',G') \iff F\subseteq F' \iff G'\subseteq G.\]
Meets and joins are computed by intersecting extents and intents
respectively, then re-closing:
\[
	\begin{aligned}
		\bigwedge_i(F_i,G_i) & =\paren{\bigcap_i F_i,R^*\paren{\bigcap_i F_i}},  &
		\bigvee_i(F_i,G_i)   & =\paren{R_*\paren{\bigcap_i G_i}, \bigcap_i G_i}.
	\end{aligned}
\]

When $R\subseteq R'\subseteq C\times D$ are nested relations,
there are two canonical ways to transport a formal concept of $R$
to one of~$R'$: re-close the extent, or re-close the intent.
\begin{proposition}\label{prop:concept-transport}
	Let $R\subseteq R'\subseteq C\times D$ be relations on sets of
	objects $C$ and attributes $D$, with extensions to power sets
	$R^*\colon \mathcal P(C)\to \mathcal P(D)$ and $R_*\colon \mathcal P(D)\to \mathcal P(C)$, and similarly for $R'$.
	Write $L=\Nuc(R)$ and $L'=\Nuc(R')$ for their concept lattices.  The maps
	\[T^{\mathrm{ext}}_{R\to R'}, T^{\mathrm{int}}_{R\to R'}:L \to L'\] defined by
	\begin{align*}
		T^{\mathrm{ext}}_{R\to R'}(F,G)
		 & :=\left(R'_*(R'^*F),\,R'^*F\right),
		\\
		T^{\mathrm{int}}_{R\to R'}(F,G)
		 & :=\left(R'_*G,\,R'^*(R'_*G)\right).
	\end{align*}
	are monotone.  Moreover:
	\begin{enumerate}[label=(\alph*)]
		\item $T^{\mathrm{ext}}_{R\to R'}(F,G)$ is the \emph{least} concept of $L'$ whose extent contains $F$.
		\item $T^{\mathrm{int}}_{R\to R'}(F,G)$ is the \emph{greatest} concept of $L'$ whose intent contains $G$.
		\item One has the inequality
		      $T^{\mathrm{ext}}_{R\to R'}(F,G)\le T^{\mathrm{int}}_{R\to R'}(F,G)$
		      in~$L'$.
	\end{enumerate}
\end{proposition}

\begin{proof}
	Both pairs are concepts of $R'$ by construction, and monotonicity follows from the monotonicity of the closure operators $R'_*R'^*$ on extents and $R'^*R'_*$ on intents.

	For \textup{(a)}, extensivity of $R'_*R'^*$ gives
	$F\subseteq R'_*(R'^*F)$.
	Now let $(F',G')\in L'$ with $F\subseteq F'$.  Since $R'^*$ is antitone,
	$G'=R'^*F'\subseteq R'^*F$,
	and applying the antitone map $R'_*$ gives
	$R'_*(R'^*F)\subseteq R'_*(G')=F'$.
	Thus $T^{\mathrm{ext}}_{R\to R'}(F,G)$ is the least concept of $L'$ whose extent contains $F$.

	For \textup{(b)}, extensivity of $R'^*R'_*$ gives
	$G\subseteq R'^*(R'_*G)$.
	If $(F',G')\in L'$ and $G\subseteq G'$, then antitonicity of $R'_*$ yields
	$F'=R'_*(G')\subseteq R'_*(G)$,
	so $(F',G')\le T^{\mathrm{int}}_{R\to R'}(F,G)$.

	For \textup{(c)}, since $(F,G)\in L$ one has $G=R^*F$.
	Because $R\subseteq R'$, every attribute related by $R$ to every element of $F$ is also related by $R'$, so
	$G=R^*F\subseteq R'^*F$.
	Applying $R'_*$ gives
	$R'_*(R'^*F)\subseteq R'_*(G)$,
	which is the extent inequality
	$T^{\mathrm{ext}}_{R\to R'}(F,G)\le T^{\mathrm{int}}_{R\to R'}(F,G)$.
\end{proof}

\subsection{Chamberwise lattice towers}\label{subsec:chamberwise}

Return to the $\Rbar$-profunctor $M$ on finite discrete sets
$\cat C$ and $\cat D$, and fix a nucleus point
$(f,g)\in\Nuc(M)$.
For $\varepsilon\ge 0$, the sublevel set of the gap matrix
defines a Boolean relation.

\begin{definition}\label{def:lattice-tower}
	For $\varepsilon\ge 0$, define a Boolean profunctor
	$R^{(f,g)}_\varepsilon\colon
	\cat C^{\op}\otimes\cat D\to\set{0,1}$ by
	\[
		R^{(f,g)}_\varepsilon(c,d)=
		\begin{cases}
			1 & \text{if }\delta^{(f,g)}(c,d)\le \varepsilon, \\
			0 & \text{if }\delta^{(f,g)}(c,d)>\varepsilon.
		\end{cases}
	\]
	Let $L_\varepsilon(f,g):=\Nuc(R^{(f,g)}_\varepsilon)$ denote
	its concept lattice.
\end{definition}

An element of $L_\varepsilon(f,g)$ is a pair $(F,G)$ with
$F=(R_\varepsilon)_*G$ and $G=(R_\varepsilon)^*F$.
The Events Theorem gives these threshold relations a direct
metric interpretation.

\begin{proposition}\label{prop:threshold-distance}
	For every $(f,g)\in\Nuc(M)$, every $\varepsilon\ge 0$, and
	every $(c,d)\in\cat C\times\cat D$,
	\[
		(c,d)\in R^{(f,g)}_\varepsilon
		\iff
		\delta^{(f,g)}(c,d)\le \varepsilon
		\iff
		d_{\pnuc}\paren{(f,g),\mathcal E_{c,d}}\le \varepsilon.
	\]
	Thus the threshold relation records exactly which event loci
	lie within projective distance~$\varepsilon$ of~$(f,g)$.
\end{proposition}

\begin{proof}
	The first equivalence is the definition of
	$R^{(f,g)}_\varepsilon$, and the second is
	Theorem~\ref{thm:events}.
\end{proof}

For brevity, write $R_\varepsilon:=R^{(f,g)}_\varepsilon$ and
$L_\varepsilon:=L_\varepsilon(f,g)$.
If $\varepsilon\le\varepsilon'$ then
$R_\varepsilon\subseteq R_{\varepsilon'}$, and
Proposition~\ref{prop:concept-transport} provides two monotone
transport maps
\[
	T^{\mathrm{ext}}_{\varepsilon,\varepsilon'}
	\;\le\;
	T^{\mathrm{int}}_{\varepsilon,\varepsilon'}
	\colon L_\varepsilon\to L_{\varepsilon'},
\]
the first preserving extents minimally, the second preserving
intents maximally.

For a fixed numerical value of $\varepsilon$, the relation
$R^{(f,g)}_\varepsilon$ need not be locally constant as $(f,g)$
varies in an order chamber, since the gap values
$\delta^{(f,g)}(c,d)$ move continuously.
What \emph{is} locally constant on an order chamber is the
\emph{order} in which incidences enter as $\varepsilon$
increases.
It is therefore natural to reindex the construction by the
equivalence classes of the chamber preorder.

Fix an order chamber $Q$ with total preorder $\preceq_Q$ on
$\cat C\times\cat D$, and let
\[
	E_0\prec_Q E_1\prec_Q\cdots\prec_Q E_m
\]
be the equivalence classes, ordered from smallest to largest.
For $0\le k\le m$ define
\[
	R^Q_k:=\bigcup_{i=0}^k E_i\subseteq \cat C\times\cat D,
\]
giving a finite chain of relations
$R^Q_0\subseteq R^Q_1\subseteq\cdots\subseteq R^Q_m$.
Setting $L^Q_k:=\Nuc(R^Q_k)$, each inclusion
$R^Q_k\subseteq R^Q_{k'}$ carries a pair of monotone transport
maps
$T^{\mathrm{ext}}_{k,k'}\le T^{\mathrm{int}}_{k,k'}\colon
L^Q_k\to L^Q_{k'}$
from Proposition~\ref{prop:concept-transport}, giving a finite
tower of concept lattices
\[
	L^Q_0\to L^Q_1\to\cdots\to L^Q_m
\]
with canonical lower and upper structure maps at each step.

\begin{proposition}\label{prop:chamberwise-tower}
	Let $Q$ be an order chamber with preorder classes
	$E_0\prec_Q\cdots\prec_Q E_m$.
	For every $(f,g)\in Q$ and every $\varepsilon\ge 0$, there is
	a unique index $k$ such that
	$R^{(f,g)}_\varepsilon=R^Q_k$.
	Consequently the real-parameter family
	$\varepsilon\mapsto L_\varepsilon(f,g)$ factors through the
	finite tower $L^Q_0\to\cdots\to L^Q_m$, and this tower
	depends only on the chamber~$Q$, not on the chosen point.
\end{proposition}

\begin{proof}
	Points of the same order chamber determine the same total
	preorder on the gap values.
	The sublevel condition
	$\delta^{(f,g)}(c,d)\le\varepsilon$ therefore selects an
	initial segment of the preorder classes, and every initial
	segment arises for a unique threshold value.
\end{proof}

\subsection{The constructible sheaf of
lattices}\label{subsec:constructible-sheaf}

The chamberwise towers assemble into a global structure over the
order-chamber complex.

\begin{proposition}\label{prop:face-specialization}
	Let $Q'$ be a face of the closure of an order chamber~$Q$.
	Then $\preceq_{Q'}$ is obtained from $\preceq_Q$ by merging
	consecutive equivalence classes, so each class
	for~$\preceq_{Q'}$ is a union of consecutive classes
	for~$\preceq_Q$.
	Consequently, the chain $\set{R^{Q'}_\ell}$ is a subsequence
	of $\set{R^Q_k}$, and the lattice tower for~$Q'$ is a
	coarsening of the tower for~$Q$: it retains the same
	concept lattices at the merged thresholds, with structure maps
	given by the direct transport for the corresponding
	(larger) inclusion.
\end{proposition}

\begin{proof}
	Passing from $Q$ to a face imposes equalities among
	neighboring gap values while preserving their order relative
	to the remaining classes.
	Hence only consecutive preorder blocks merge, and the
	relation chain and lattice tower coarsen accordingly.
\end{proof}

Propositions~\ref{prop:chamberwise-tower}
and~\ref{prop:face-specialization} say that the assignment
\[
	Q\;\longmapsto\;
	\paren{L^Q_0\to L^Q_1\to\cdots\to L^Q_m}
\]
is a \emph{constructible sheaf of lattice towers} over the
order-chamber complex: each open chamber receives a finite tower
of concept lattices, and specialization to a face coarsens the
tower by merging consecutive floors.
The two transport maps
$T^{\mathrm{ext}}\le T^{\mathrm{int}}$ furnish canonical lower
and upper structure maps at each step of each tower; both are
preserved by the coarsening.

\begin{figure}[t]
	\centering
	\[
		\begin{tikzcd}[column sep=2.4em,row sep=2.2em]
			L^Q_0 \arrow[r] \arrow[d,equal] &
			L^Q_1 \arrow[r] &
			L^Q_2 \arrow[r] \arrow[d,equal] &
			L^Q_3 \arrow[r] &
			L^Q_4 \arrow[r] \arrow[d,equal] &
			L^Q_5 \arrow[d,equal] \\
			L^{Q'}_0 \arrow[rr,"T_{0,2}"] & &
			L^{Q'}_1 \arrow[rr,"T_{2,4}"] & &
			L^{Q'}_2 \arrow[r,"T_{4,5}"] &
			L^{Q'}_3
		\end{tikzcd}
	\]
	\caption{Specialization to a face.
		If $Q'\le\overline Q$ merges consecutive preorder blocks,
		the tower over~$Q'$ retains the lattices at the merged
		thresholds, with structure maps given by the direct
		transport for the coarsened inclusions.}
	\label{fig:merge-floors}
\end{figure}

We illustrate on the running example.
Let $\cat C=\{c_0,c_1,c_2\}$,
$\cat D=\{d_1,d_2,d_3,d_4\}$, and let
\[
	M=
	\begin{bmatrix}
		0.7 & 1.5  & 1.7 & -1.3 \\
		1.2 & 2.6  & 0.1 & 2.2  \\
		2.0 & -1.6 & 2.0 & -2.9
	\end{bmatrix}.
\]
Work in the gauge slice $f(c_0)=0$, so
$f=(0,x,y)\in\mathbb R^3$.
Consider three nucleus points
\[
	f_1=(0,0,0),\qquad f_2=(0,-0.1,0),\qquad f_3=(0,0.1,0),
\]
with $g_i:=M^*f_i$ and
$\delta_i(c,d):=M(c,d)-f_i(c)-g_i(d)$.
One computes
\[
	\begin{aligned}
		g_1 & =(0.7,-1.6,0.1,-2.9), \\
		g_2 & =(0.7,-1.6,0.2,-2.9), \\
		g_3 & =(0.7,-1.6,0.0,-2.9)
	\end{aligned}
\]
and gap matrices
\[
	\delta_1=
	\begin{bmatrix}
		0.0 & 3.1 & 1.6 & 1.6 \\
		0.5 & 4.2 & 0.0 & 5.1 \\
		1.3 & 0.0 & 1.9 & 0.0
	\end{bmatrix},\qquad
	\delta_2=
	\begin{bmatrix}
		0.0 & 3.1 & 1.5 & 1.6 \\
		0.6 & 4.3 & 0.0 & 5.2 \\
		1.3 & 0.0 & 1.8 & 0.0
	\end{bmatrix},\qquad
	\delta_3=
	\begin{bmatrix}
		0.0 & 3.1 & 1.7 & 1.6 \\
		0.4 & 4.1 & 0.0 & 5.0 \\
		1.3 & 0.0 & 2.0 & 0.0
	\end{bmatrix}.
\]
In particular,
\[
	\delta_2(c_0,d_3)<\delta_2(c_0,d_4),\qquad
	\delta_1(c_0,d_3)=\delta_1(c_0,d_4),\qquad
	\delta_3(c_0,d_4)<\delta_3(c_0,d_3),
\]
so $f_2$ and $f_3$ lie in adjacent order chambers $Q_2$ and
$Q_3$, separated by the wall
$Q_1=\overline{Q_2}\cap\overline{Q_3}$ through~$f_1$.
Figure~\ref{fig:order_chambers_zoom} shows these three points.

\begin{figure}[t]
	\centering
	\includegraphics[width=0.84\linewidth]{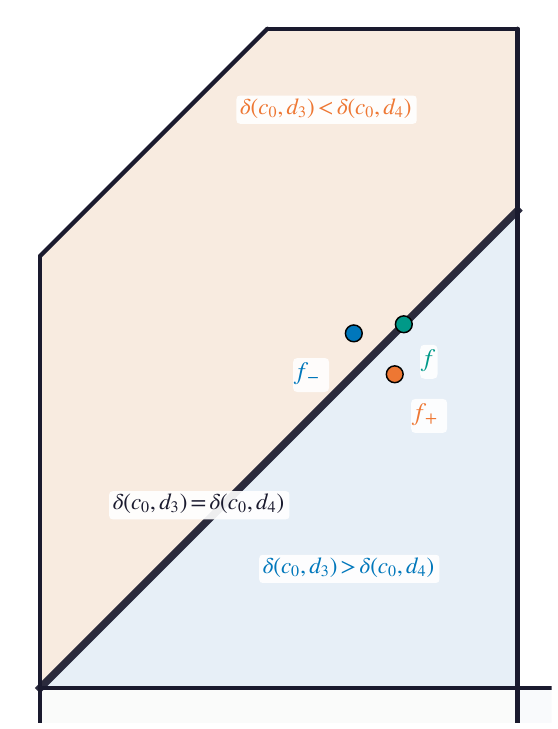}
	\caption{Local wall crossing at the $1.6$ tie.
		Along the diagonal wall one has
		$\delta(c_0,d_3)=\delta(c_0,d_4)$.
		The points $f_2$ and $f_3$ lie in the adjacent chambers
		where these two event radii are ordered oppositely, while
		$f_1$ lies on the wall.}
	\label{fig:order_chambers_zoom}
\end{figure}

Let $R^{(f_i,g_i)}_\varepsilon:=\{(c,d)\mid
\delta_i(c,d)\le\varepsilon\}$ be the threshold relations and
write $L(R):=\Nuc(R)$ for the concept lattice of a relation
$R\subseteq\cat C\times\cat D$.
Define the relations
\[
	\begin{aligned}
		R_0    & :=\{(c_0,d_1),(c_1,d_3),(c_2,d_2),(c_2,d_4)\}, \\
		R_1    & :=R_0\cup\{(c_1,d_1)\},\qquad
		R_2:=R_1\cup\{(c_2,d_1)\},                              \\
		R_{3a} & :=R_2\cup\{(c_0,d_3)\},\qquad
		R_{3b}:=R_2\cup\{(c_0,d_4)\},                           \\
		R_4    & :=R_2\cup\{(c_0,d_3),(c_0,d_4)\},              \\
		R_5    & :=R_4\cup\{(c_2,d_3)\},\qquad
		R_6:=R_5\cup\{(c_0,d_2)\},                              \\
		R_7    & :=R_6\cup\{(c_1,d_2)\},\qquad
		R_8:=R_7\cup\{(c_1,d_4)\}.
	\end{aligned}
\]
The chamberwise towers are:
\begin{align*}
	\text{at }f_2\in Q_2: &  &
	R_0\subset R_1\subset R_2\subset R_{3a}\subset R_4\subset R_5\subset R_6\subset R_7\subset R_8 \\
	\text{at }f_1\in Q_1: &  &
	R_0\subset R_1\subset R_2\subset R_4\subset R_5\subset R_6\subset R_7\subset R_8               \\
	\text{at }f_3\in Q_3: &  &
	R_0\subset R_1\subset R_2\subset R_{3b}\subset R_4\subset R_5\subset R_6\subset R_7\subset R_8.
\end{align*}
The only combinatorial difference between the two chambers is the
order in which $(c_0,d_3)$ and $(c_0,d_4)$ enter; on the wall
they enter simultaneously.

The corresponding concept lattices are:
\[
	\begin{aligned}
		L(R_2)=\{ &
		(\varnothing\mid \{d_1,d_2,d_3,d_4\}),\
		(\{c_1\}\mid \{d_1,d_3\}),                  \\
		          & (\{c_2\}\mid \{d_1,d_2,d_4\}),\
		(\{c_0,c_1,c_2\}\mid \{d_1\})\},
	\end{aligned}
\]
\[
	\begin{aligned}
		L(R_{3a})=\{ &
		(\varnothing\mid \{d_1,d_2,d_3,d_4\}),\
		(\{c_2\}\mid \{d_1,d_2,d_4\}),                 \\
		             & (\{c_0,c_1\}\mid \{d_1,d_3\}),\
		(\{c_0,c_1,c_2\}\mid \{d_1\})\},
	\end{aligned}
\]
\[
	\begin{aligned}
		L(R_{3b})=\{ &
		(\varnothing\mid \{d_1,d_2,d_3,d_4\}),\
		(\{c_1\}\mid \{d_1,d_3\}),\
		(\{c_2\}\mid \{d_1,d_2,d_4\}),                 \\
		             & (\{c_0,c_2\}\mid \{d_1,d_4\}),\
		(\{c_0,c_1,c_2\}\mid \{d_1\})\},
	\end{aligned}
\]
\[
	\begin{aligned}
		L(R_4)=\{ &
		(\varnothing\mid \{d_1,d_2,d_3,d_4\}),\
		(\{c_0\}\mid \{d_1,d_3,d_4\}),\
		(\{c_2\}\mid \{d_1,d_2,d_4\}),              \\
		          & (\{c_0,c_1\}\mid \{d_1,d_3\}),\
		(\{c_0,c_2\}\mid \{d_1,d_4\}),\
		(\{c_0,c_1,c_2\}\mid \{d_1\})\}.
	\end{aligned}
\]
For $k\ge 4$ the lattices are the same in all three towers:
$L(R_5)$ is a $3$-element chain, $L(R_6)$ and $L(R_7)$ are
$2$-element chains, and $L(R_8)$ is the one-point lattice.

In this example, the direct transport map
$T_{2,4}\colon L(R_2)\to L(R_4)$ agrees with both composites
$T_{3a,4}\circ T_{2,3a}$ and $T_{3b,4}\circ T_{2,3b}$, for
either choice of $T^{\mathrm{ext}}$ or $T^{\mathrm{int}}$.
Figure~\ref{fig:wall-specialization-diamond} displays this as a
commutative diamond.

\begin{figure}[t]
	\centering
	\[
		\begin{tikzcd}[row sep=2.7em,column sep=3.3em]
			& L(R_{3a}) \arrow[dr,"T_{3a,4}"] & \\
			L(R_2) \arrow[ur,"T_{2,3a}"]
			\arrow[rr,"T_{2,4}" description]
			\arrow[dr,"T_{2,3b}"'] & & L(R_4) \\
			& L(R_{3b}) \arrow[ur,"T_{3b,4}"'] &
		\end{tikzcd}
	\]
	\caption{Wall specialization at the $1.6$ tie.
		The adjacent chambers $Q_2$ and $Q_3$ correspond to the
		two strict refinements of the wall preorder:
		in $Q_2$ the incidence $(c_0,d_3)$ enters before
		$(c_0,d_4)$ (intermediate relation $R_{3a}$), while in
		$Q_3$ the order is reversed ($R_{3b}$).
		On the wall, the tie merges these floors, yielding the
		direct inclusion $R_2\subset R_4$.
		In this example the diamond commutes for both
		$T^{\mathrm{ext}}$ and $T^{\mathrm{int}}$.}
	\label{fig:wall-specialization-diamond}
\end{figure}

\begin{figure}[t]
	\centering
	\includegraphics[width=0.98\linewidth]{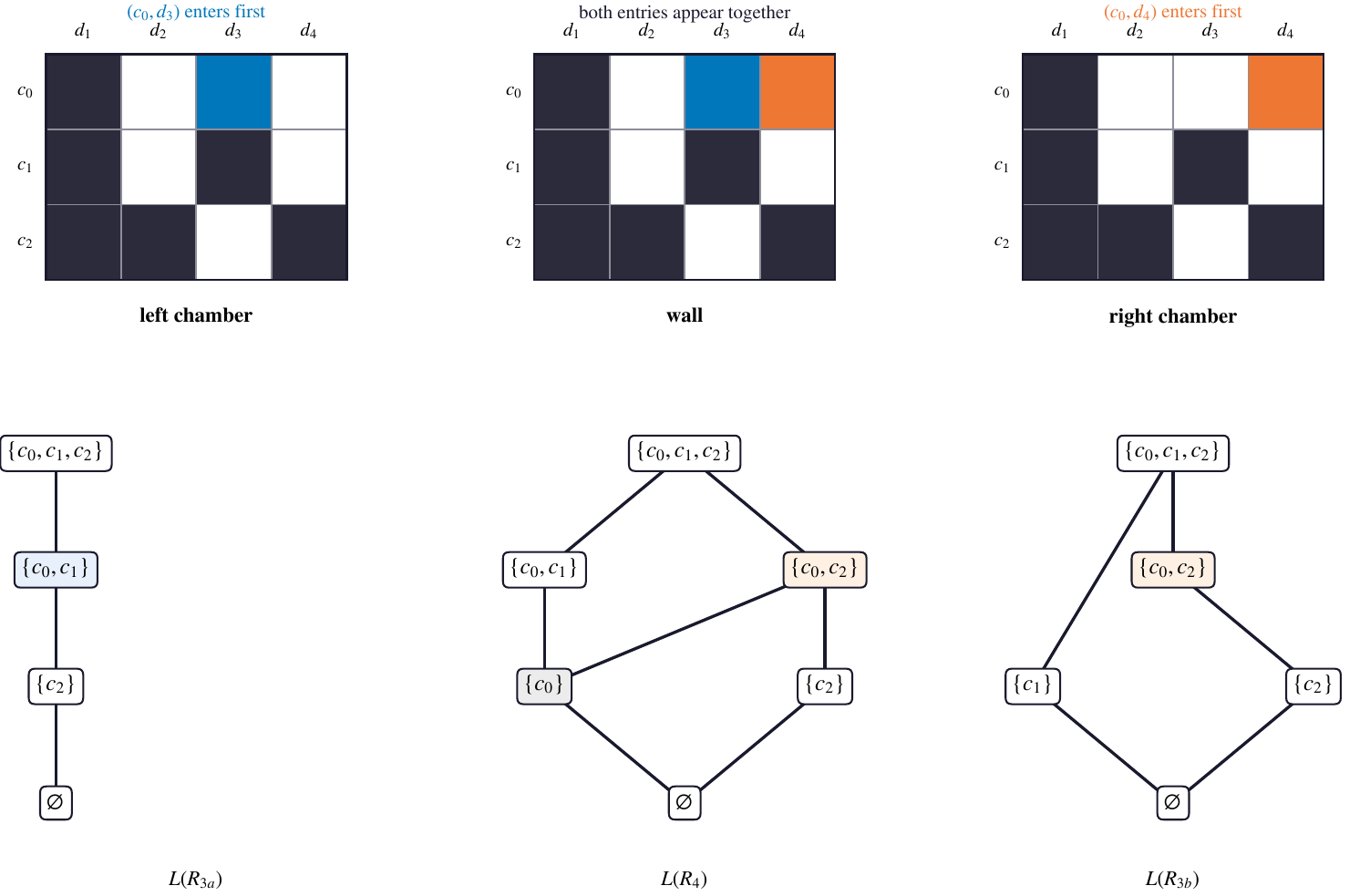}
	\caption{Thresholded relations and Boolean nuclei across
		the $1.6$ tie.
		In the left chamber the new incidence $(c_0,d_3)$ enters
		first, giving $R_{3a}$; on the wall both incidences
		appear together, giving $R_4$; in the right chamber
		$(c_0,d_4)$ enters first, giving $R_{3b}$.
		The lattices below show the corresponding change in the
		Boolean nuclei.}
	\label{fig:threshold-triptych}
\end{figure}

\section{The square case: optimal assignments and Chebyshev
centering}\label{sec:centering}

When $\cat C=\cat D$ and $|\cat C|=n$, the projective nucleus
of a real $n\times n$ matrix $M$ acquires canonical structure
that has no analogue in the rectangular case:
a distinguished scalar $\operatorname{val}(M)$ (the tropical
determinant), at most one full-dimensional cell (and generically
exactly one), and within that cell a canonical center point
whose inscribed radius is a computable invariant of~$M$.
This section develops these invariants:
\S\ref{subsec:square-case}--\S\ref{subsec:admissible-permutations}
identify the full-dimensional cell via the tropical determinant,
and
\S\ref{subsec:chebyshev-lp}--\S\ref{subsec:centering-proofs}
solve the Chebyshev centering problem, expressing the optimal
radius as a minimum directed cycle mean.

\subsection{Full-dimensional cells are permutation
cells}\label{subsec:square-case}

The projective space $\pcat C$ has dimension $n-1$, so every
witness polyhedron $\Cell(Y)$ has dimension at most $n-1$.
We call $\Cell(Y)$ \emph{full-dimensional} if it achieves this
bound.
Any subset $Y\subseteq\cat C\times\cat C$ that covers both
factors has $|Y|\ge n$, with equality if and only if $Y$ is the
graph of a permutation $\sigma\in S_n$:
\[
	\Gamma_\sigma:=\set{(c,\sigma(c))\mid c\in\cat C}.
\]

\begin{proposition}\label{prop:topcell-permutation}
	If $\Cell(Y)$ is full-dimensional, then $Y=\Gamma_\sigma$
	for a permutation $\sigma\in S_n$.
\end{proposition}

\begin{proof}
	Suppose $Y$ contains two pairs $(c,d)$ and $(c',d)$ with
	the same second coordinate~$d$.
	The witness equalities~\eqref{eq:Cell-eq} give
	$f(c)+g(d)=M(c,d)$ and $f(c')+g(d)=M(c',d)$, hence
	$f(c')-f(c)=M(c',d)-M(c,d)$ throughout $\Cell(Y)$.
	In the gauge $f(c_0)=0$, this is a nontrivial affine
	relation among the $n-1$ free coordinates of~$f$, so
	$\Cell(Y)$ has dimension at most~$n-2$.
	Therefore each $d\in\cat C$ appears in at most one pair
	of~$Y$.
	Since $Y$ covers $\cat D=\cat C$, each $d$ appears exactly
	once, giving $|Y|\le n$.
	But $Y$ also covers $\cat C$, so $|Y|\ge n$, and $Y$ is
	the graph of a bijection.
\end{proof}

\subsection{The tropical determinant and the full-dimensional
cell}\label{subsec:admissible-permutations}

For $\sigma\in S_n$, the $M$-\emph{cost} of~$\sigma$ is
$\sum_{c\in\cat C}M(c,\sigma(c))$.
The \emph{tropical determinant} of~$M$ is the minimum cost:
\[
	\operatorname{val}(M)
	:=\min_{\sigma\in S_n}\sum_{c\in\cat C}
	M(c,\sigma(c)).
\]
The matrix $M$ is \emph{tropically nonsingular} if the minimum
is achieved by a unique permutation.
A permutation achieving $\operatorname{val}(M)$ is
\emph{optimal}; finding one is the linear assignment
problem~\cite{Kuhn1955,BurkardDellAmicoMartello2009};
see also~\cite{MaclaganSturmfels2015ITG} for the tropical
perspective.

By Proposition~\ref{prop:topcell-permutation}, every
full-dimensional cell has the form $\Cell(\Gamma_\sigma)$.
The admissible permutations turn out to be exactly the optimal
ones:

\begin{proposition}\label{prop:admissible-permutations}
	The permutation graph $\Gamma_\sigma$ is admissible if and
	only if $\sigma$ achieves $\operatorname{val}(M)$.
\end{proposition}

This is a restatement, in Isbell-duality language, of the
classical LP-duality characterization of optimal
assignments~\cite[Ch.~17]{Schrijver2003}.
We include the proof here because it 
gives an explicit construction of a nucleus point
in~$\Cell(\Gamma_\sigma)$.  This
is useful because the full-dimensional cell is difficult to reach by
naive geometric means: projecting a generic point of~$\pcat C$
onto~$\pnuc(M)$ typically lands on a low-dimensional boundary
cell rather than the full-dimensional interior.
The shortest-path construction in the proof 
reappears, in strengthened form,
in the Centering Theorem.

\begin{proposition}\label{prop:admissible-permutations}
	The permutation graph $\Gamma_\sigma$ is admissible if and
	only if $\sigma$ achieves $\operatorname{val}(M)$.
\end{proposition}

\begin{proof}
	Suppose $\Gamma_\sigma$ is admissible and choose
	$(f,g)\in\Cell(\Gamma_\sigma)$.
	Summing the witness equalities
	$f(c)+g(\sigma(c))=M(c,\sigma(c))$ over~$c$ gives
	\begin{equation}\label{eq:dual-value}
		\sum_{c}f(c)+\sum_{d}g(d)
		=\sum_{c}M(c,\sigma(c)),
	\end{equation}
	since $\sigma$ is a bijection.
	For any other permutation~$\tau$, summing the Isbell
	inequalities $f(c)+g(\tau(c))\le M(c,\tau(c))$ over~$c$
	gives
	$\sum_c f(c)+\sum_d g(d)\le\sum_c M(c,\tau(c))$.
	Comparing with~\eqref{eq:dual-value} shows $\sigma$
	achieves $\operatorname{val}(M)$.

	Now assume $\sigma$ achieves $\operatorname{val}(M)$.
	We construct $(f,g)\in\Cell(\Gamma_\sigma)$ via
	shortest-path potentials.

	\emph{Step~1: the auxiliary digraph.}\enspace
	Form the complete directed graph on $\cat C$ with edge
	weights
	\[
		w(c\to c')
		:=M(c',\sigma(c))-M(c,\sigma(c)).
	\]
	The weight $w(c\to c')$ measures the change in cost when
	column $\sigma(c)$ is reassigned from row $c$ to row~$c'$.

	\emph{Step~2: no negative cycles.}\enspace
	For a directed cycle $x_0\to x_1\to\cdots\to x_k=x_0$,
	the edge weights telescope:
	\[
		\sum_{i=0}^{k-1}w(x_i\to x_{i+1})
		=\sum_{i=0}^{k-1}\bigl[M(x_{i+1},\sigma(x_i))
		-M(x_i,\sigma(x_i))\bigr].
	\]
	If this were negative, replacing $\sigma$ on the cycle by
	$\tau(x_{i+1})=\sigma(x_i)$ would yield a permutation of
	strictly lower cost, contradicting optimality.

	\emph{Step~3: shortest-path potentials.}\enspace
	Fix a base vertex $c_0$.
	Since there are no negative cycles, shortest-path distances
	from $c_0$ are well-defined.
	Set $f(c)$ to be the shortest-path distance from $c_0$
	to~$c$, so $f(c_0)=0$ and
	$f(c')\le f(c)+w(c\to c')$ for all~$c,c'$.
	Define $g$ by
	\[
		g(\sigma(c)):=M(c,\sigma(c))-f(c).
	\]
	The witness equalities
	$f(c)+g(\sigma(c))=M(c,\sigma(c))$ hold by construction.
	For the off-diagonal Isbell inequality at
	$(c',\sigma(c))$ with $c'\neq c$:
	\[
		f(c')+g(\sigma(c))
		=f(c')+M(c,\sigma(c))-f(c)
		\le M(c',\sigma(c)),
	\]
	since $M(c',\sigma(c))-M(c,\sigma(c))=w(c\to c')
	\ge f(c')-f(c)$.
	As $\sigma$ is a bijection, every $d\in\cat C$ is
	$\sigma(c)$ for a unique~$c$, so
	$f(c')+g(d)\le M(c',d)$ for all pairs.
	Thus $(f,g)\in\Cell(\Gamma_\sigma)$.
\end{proof}

\begin{proposition}\label{prop:unique-optimal}
	If $\sigma$ and $\tau$ are distinct permutations both
	achieving $\operatorname{val}(M)$, then
	$\Cell(\Gamma_\sigma)=\Cell(\Gamma_\tau)
	=\Cell(\Gamma_\sigma\cup\Gamma_\tau)$,
	and this cell has dimension at most~$n-3$.
\end{proposition}

\begin{proof}
	Let $(f,g)\in\Cell(\Gamma_\sigma)$.
	Summing the witness equalities for~$\sigma$ gives
	$\sum_c f(c)+\sum_d g(d)=\operatorname{val}(M)$
	as in~\eqref{eq:dual-value}.
	Summing the Isbell inequalities along~$\tau$ gives
	$\sum_c f(c)+\sum_d g(d)\le\operatorname{val}(M)$,
	using that $\tau$ is a bijection achieving the same value.
	The left-hand sides are identical, so every inequality
	$f(c)+g(\tau(c))\le M(c,\tau(c))$ is an equality.
	Thus $(f,g)\in\Cell(\Gamma_\tau)$, giving
	$\Cell(\Gamma_\sigma)\subseteq\Cell(\Gamma_\tau)$;
	by symmetry the two cells coincide.
	Since $\sigma\neq\tau$ differ on at least two elements
	of~$\cat C$, we have
	$|\Gamma_\sigma\cup\Gamma_\tau|\ge n+2$, so the cell
	satisfies at least $n+2$ witness equalities and has
	dimension at most~$n-3$.
\end{proof}

Combining the three propositions: $\pnuc(M)$ has a
full-dimensional cell if and only if $M$ is tropically
nonsingular, and in that case there is exactly one.

\subsection{Chebyshev centering}\label{subsec:centering-setup}

Assume now that $M$ is tropically nonsingular with unique
optimal permutation~$\sigma$.
On the full-dimensional cell $\Cell(\Gamma_\sigma)$, the
witness equalities $f(c)+g(\sigma(c))=M(c,\sigma(c))$
pin $g$ to~$f$:
\begin{equation}\label{eq:g-from-f-sigma}
	g(\sigma(c))=M(c,\sigma(c))-f(c),\qquad c\in\cat C.
\end{equation}
The gap matrix on this cell is
\[
	\delta^{(f,g)}(c,d)=M(c,d)-f(c)-g(d),
\]
with $\delta^{(f,g)}(c,\sigma(c))=0$ for all $c$ and
$\delta^{(f,g)}(c,d)>0$ for $d\ne\sigma(c)$ at every interior
point of $\Cell(\Gamma_\sigma)$.
The Events Theorem identifies each positive gap entry with
the projective distance to the corresponding event locus.
List the $n(n-1)$ positive gap values in nondecreasing order:
\[
	e_1(f,g)\le e_2(f,g)\le\cdots\le e_{n(n-1)}(f,g).
\]
Then $e_1$ is the distance from $(f,g)$ to the nearest cell
wall.
A \emph{Chebyshev center} of $\Cell(\Gamma_\sigma)$ is any
point maximizing~$e_1$, and the \emph{Chebyshev radius} is
\[
	r^*:=\max_{(f,g)\in\Cell(\Gamma_\sigma)}e_1(f,g).
\]
Since $e_1=0$ on $\partial\Cell(\Gamma_\sigma)$ and $e_1>0$ in
the interior, every Chebyshev center lies in the open cell
$\Cell^\circ(\Gamma_\sigma)$.

\subsection{The Chebyshev LP}\label{subsec:chebyshev-lp}

Substituting~\eqref{eq:g-from-f-sigma} into the gap formula,
the gap at an off-diagonal pair $(c,\sigma(c'))$ with $c\ne c'$
is
\begin{equation}\label{eq:gap-reduced}
	\delta^{(f,g)}(c,\sigma(c'))
	=\alpha(c,c')+f(c')-f(c),
\end{equation}
where
\begin{equation}\label{eq:alpha-def}
	\alpha(c,c')
	:=M(c,\sigma(c'))-M(c',\sigma(c'))
\end{equation}
depends only on $M$ and $\sigma$.
Note that $\alpha(c,c')=w(c'\to c)$, the edge weight of the
auxiliary digraph from the proof of
Proposition~\ref{prop:admissible-permutations} with the
direction reversed: the shortest-path potentials constructed
there satisfy $\alpha(c,c')+f(c')-f(c)\ge 0$, which is exactly
the condition $\delta\ge 0$ on the off-diagonal gaps.

Since $\sigma$ is a bijection, the off-diagonal pairs $(c,d)$
with $d\ne\sigma(c)$ are exactly the pairs $(c,\sigma(c'))$
with $c\ne c'$.
In the gauge slice $f(c_0)=0$, the Chebyshev problem becomes
the linear program
\begin{equation}\label{eq:chebyshev-lp}
	\begin{aligned}
		\text{maximize}\quad   & r
		\\
		\text{subject to}\quad
		                       & \alpha(c,c')+f(c')-f(c)\ge r,
		\quad c\ne c',
		\\
		                       & f(c_0)=0.
	\end{aligned}
\end{equation}
This LP has $n$ variables---$f(c_1),\ldots,f(c_{n-1})$ and
$r$---and $n(n-1)$ inequality constraints.

\subsection{Statement and proofs}\label{subsec:centering-proofs}

\begin{theorem}[The Centering
Theorem]\label{thm:centering}
	Let $M$ be a real $n\times n$ matrix with a unique optimal
	permutation~$\sigma$, and let
	$\alpha(c,c')=M(c,\sigma(c'))-M(c',\sigma(c'))$.
	\begin{enumerate}
		\item[\textup{(a)}]
		      The Chebyshev radius $r^*$ equals the minimum
		      directed cycle mean of the digraph
		      $(\cat C,\alpha)$:
		      \[
			      r^*=\min_\gamma\frac{1}{|\gamma|}
			      \sum_{(c,c')\in\gamma}\alpha(c,c'),
		      \]
		      where $\gamma$ ranges over all directed cycles in
		      the complete graph on~$\cat C$.
		\item[\textup{(b)}]
		      At every vertex of the \textup{(}convex\textup{)}
		      set of Chebyshev centers, the multiplicity of
		      $e_1$ is at least~$n$.
		\item[\textup{(c)}]
		      For $M$ outside a proper algebraic subset of
		      $\R^{n\times n}$, the Chebyshev center is unique
		      and the multiplicity of $e_1$ is exactly~$n$.
	\end{enumerate}
\end{theorem}

\begin{proof}[Proof of \textup{(a)}]
	The constraints in~\eqref{eq:chebyshev-lp} read
	$f(c)-f(c')\le\alpha(c,c')-r$ for all $c\ne c'$.
	This is a system of \emph{difference constraints}: upper
	bounds on pairwise differences of the variables~$f(c)$.
	By a standard result in combinatorial optimization
	\cite[Corollary~8.3b]{Schrijver2003}, such a system is
	feasible if and only if the weighted digraph with edge
	weights $\alpha(c,c')-r$ has no negative-weight directed
	cycle.
	A cycle $\gamma$ of length $|\gamma|$ has nonnegative total
	weight if and only if
	$\sum_{(c,c')\in\gamma}(\alpha(c,c')-r)\ge 0$,
	equivalently
	$r\le(1/|\gamma|)\sum_{(c,c')\in\gamma}\alpha(c,c')$,
	the cycle mean of~$\gamma$.
	This holds for every directed cycle if and only if
	$r\le\min_\gamma\operatorname{mean}(\gamma)$.
	Hence $r^*=\min_\gamma\operatorname{mean}(\gamma)$.
\end{proof}

\begin{proof}[Proof of \textup{(b)}]
	The LP~\eqref{eq:chebyshev-lp} lives in $\R^n$ with
	$n(n-1)$ inequality constraints.
	The set of optimal solutions is a face of the feasible
	polyhedron.
	At any vertex of this face, at least $n$ linearly
	independent inequality constraints are active.
	Each active constraint
	$\alpha(c,c')+f^*(c')-f^*(c)=r^*$ corresponds,
	via~\eqref{eq:gap-reduced}, to an off-diagonal pair
	$(c,\sigma(c'))$ with
	$\delta^{(f^*,g^*)}(c,\sigma(c'))=r^*=e_1$.
	Hence the multiplicity of~$e_1$ is at least~$n$.
\end{proof}

\begin{proof}[Proof of \textup{(c)}]
	LP nondegeneracy---the condition that no vertex of the
	feasible polyhedron has more than $n$ active inequality
	constraints---fails on a proper algebraic subset of the
	space of $\alpha$-coefficients, hence of $\R^{n\times n}$.
	For nondegenerate instances, the optimal vertex is unique
	and exactly $n$ constraints are active.
\end{proof}

\subsection{A square example}\label{subsec:square-example}

The running $3\times 4$ example is not square.
To illustrate the Centering Theorem, consider the
$3\times 3$ matrix
\[
	M=
	\begin{bmatrix}
		1 & 1 & 6 \\
		6 & 3 & 1 \\
		1 & 6 & 5
	\end{bmatrix}.
\]
The unique optimal permutation is the $3$-cycle
$\sigma=(1\,2\,0)$, with
$\operatorname{val}(M)=M(0,1)+M(1,2)+M(2,0)=1+1+1=3$,
so $M$ is tropically nonsingular and
$\Cell(\Gamma_\sigma)$ is the unique full-dimensional cell.
On this cell the three witness equalities
$f(c)+g(\sigma(c))=M(c,\sigma(c))$ pin
$g$ to~$f$; in the gauge $f(c_0)=0$ with
$f_1=f(c_1)$, $f_2=f(c_2)$, the six off-diagonal
gap functions become
\begin{gather*}
	\delta(0,0) = f_2,\qquad
	\delta(1,1) = 2-f_1,\qquad
	\delta(2,1) = 5-f_2,\\
	\delta(0,2) = 5+f_1,\qquad
	\delta(1,0) = 5-f_1+f_2,\qquad
	\delta(2,2) = 4+f_1-f_2.
\end{gather*}
Requiring each to be non-negative gives the cell.
Four of the six inequalities are active:
\[
	f_2\ge 0,\qquad
	f_1\le 2,\qquad
	f_2\le 5,\qquad
	f_2-f_1\le 4;
\]
the remaining two ($f_1\ge -5$ and $f_1-f_2\le 5$)
are redundant.
The cell is therefore a quadrilateral with vertices
$(-4,0)$, $(2,0)$, $(2,5)$, $(1,5)$.

\begin{figure}[ht]
	\centering
	\includegraphics[width=0.88\linewidth]{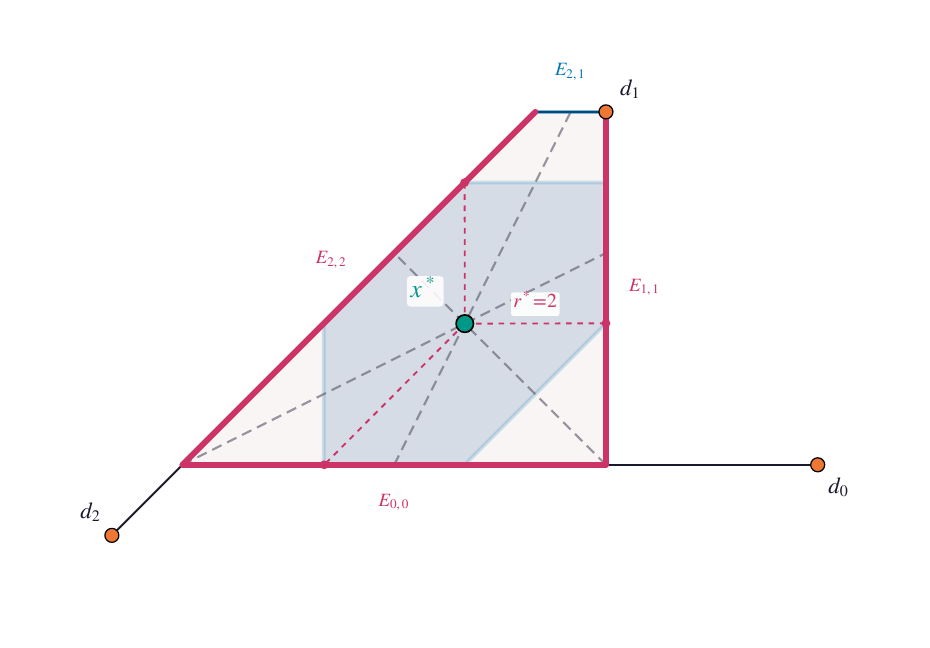}
	\caption{The Chebyshev center of
		$\Cell(\Gamma_\sigma)$ for the $3\times 3$
		matrix~$M$ of \S\ref{subsec:square-example},
		shown in the gauge $f(c_0)=0$.
		The full-dimensional cell (shaded quadrilateral)
		has four walls, one for each active Isbell
		inequality.
		Three walls (bold, $E_{0,0}$, $E_{1,1}$, $E_{2,2}$)
		are tight at the center~$x^*$, meaning
		$\delta=r^*=2$; the fourth
		($E_{2,1}$, lighter) has $\delta=3>r^*$.
		The inscribed hexagon is the spread-metric
		ball of radius~$r^*$.
		Dashed lines from~$x^*$ to the three tangent
		points confirm equidistance.
		Three order-chamber walls through~$x^*$
		partition the cell into regions of constant
		gap ordering.
		Column anchors $d_0$, $d_1$, $d_2$ are the
		images of the columns of~$M$ in~$\TP^2$.}
	\label{fig:chebyshev_center}
\end{figure}

The $\alpha$-matrix
$\alpha(c,c')=M(c,\sigma(c'))-M(c',\sigma(c'))$ is
\[
	\alpha=
	\begin{bmatrix}
		\cdot & 5 & 0 \\
		2 & \cdot & 5 \\
		5 & 4 & \cdot
	\end{bmatrix}.
\]
The five directed cycle means are:
\[
	\begin{aligned}
		\mu(0\to 1\to 0) &= \tfrac{5+2}{2}=3.5,\\
		\mu(0\to 2\to 0) &= \tfrac{0+5}{2}=2.5,\\
		\mu(1\to 2\to 1) &= \tfrac{5+4}{2}=4.5,
	\end{aligned}
	\qquad
	\begin{aligned}
		\mu(0\to 1\to 2\to 0) &= \tfrac{5+5+5}{3}=5,\\
		\mu(0\to 2\to 1\to 0) &= \tfrac{0+4+2}{3}=2.
	\end{aligned}
\]
The minimum is $r^*=2$, achieved uniquely by the
$3$-cycle $0\to 2\to 1\to 0$.
Solving the difference-constraint system
$f(c)-f(c')\le\alpha(c,c')-r^*$ via shortest-path
potentials gives the Chebyshev center
$f^*=(0,0,2)$, $g^*=(-1,1,1)$.
The gap matrix at the center is
\[
	\delta^{(f^*,g^*)}=
	\begin{bmatrix}
		2 & 0 & 5 \\
		7 & 2 & 0 \\
		0 & 3 & 2
	\end{bmatrix},
\]
whose zero entries sit at the $\sigma$-positions
$(0,1)$, $(1,2)$, $(2,0)$.
As predicted by part~(b) of the Centering Theorem,
exactly $n=3$ off-diagonal gaps achieve the value
$r^*=2$: these are
$\delta(0,0)$, $\delta(1,1)$, and~$\delta(2,2)$,
corresponding to the event loci $E_{0,0}$, $E_{1,1}$,
and~$E_{2,2}$.
The center is equidistant from these three walls
in the Hilbert projective metric, and the inscribed
hexagonal ball of radius~$r^*$ is tangent to each
(Figure~\ref{fig:chebyshev_center}).

\section{Conclusion}\label{sec:conclusion}

A real matrix $M$ is a coordinate presentation of an intrinsic geometric object: its Isbell nucleus.
The projective nucleus $\pnuc(M)$ carries a Hilbert projective metric and a witness-cell decomposition, both arising directly from the enriched-categorical structure, and both invariant under the external gauge transformations that express the freedom in choosing a coordinate presentation.
The gap matrix $\delta^{(f,g)}$ ties these two geometries together pointwise.

\subsection{What the gap matrix organizes}
For a nucleus point $(f,g)$, the gap matrix
$\delta^{(f,g)}(c,d)=M(c,d)-f(c)-g(d)$
is a nonnegative matrix that simultaneously plays three roles.
Its zero pattern determines the witness cell containing $(f,g)$.
Its positive entries are the exact projective distances to the event loci where new witnesses appear (the Events Theorem).
And thresholding it at successive radii extracts a tower of formal concept lattices, capturing the discrete combinatorial structure visible from $(f,g)$ at each resolution.

The gap matrix can therefore be regarded as a local coordinate system for the polyhedral geometry, expressed in the projective metric, that at the same time organizes the lattice-theoretic shadow of the surrounding cell structure.

\subsection{Discrete algebraic structure from continuous geometry}
The lattice-tower construction in Section~\ref{subsec:polyhedral_pnuc} extracts discrete algebraic objects---formal concept lattices with joins, meets, and Galois connections---from the continuous projective geometry of $\pnuc(M)$.
The key point is that the correct thresholding is not performed on the matrix $M$ directly, but on the gap matrix $\delta^{(f,g)}$ at a chosen basepoint.
This pointed thresholding is geometrically meaningful: it records which event loci lie within a given projective radius.
Proposition~\ref{prop:threshold-distance} makes this precise, and Proposition~\ref{prop:chamberwise-tower} shows that on each order chamber the real-parameter family factors through a finite tower of concept lattices with canonical floor mergers.
Moving to a face of the chamber complex merges consecutive floors, producing canonical specialization maps (Proposition~\ref{prop:face-specialization}).

The order-chamber complex thus carries a constructible sheaf of lattice towers, determined by the projective geometry of $\pnuc(M)$, in which each lattice encodes the Galois-closed dependencies among row--column incidences visible at a given resolution.

\subsection{Chebyshev centering and minimum cycle means}
The Centering Theorem (Theorem~\ref{thm:centering}) adds a further layer to the interaction between metric and polyhedral structure.
While the Events Theorem measures how far a given point is from each cell wall, the Centering Theorem identifies the point that maximizes the minimum such distance.
The optimal radius turns out to equal the minimum directed cycle mean of an edge-weighted digraph derived from the matrix, connecting the projective geometry of the nucleus to the classical theory of optimal assignment~\cite{Kuhn1955,BurkardDellAmicoMartello2009} and to Karp's characterization of minimum cycle means~\cite{Karp1978}.
At the Chebyshev center, the smallest gap $e_1$ is achieved with multiplicity at least $n$---generically exactly $n$---so the center is equidistant from $n$ event loci, analogous to the incenter of a simplex.
The pair sums $S(c,c')=M(c,\sigma(c'))+M(c',\sigma(c))-M(c,\sigma(c))-M(c',\sigma(c'))$, being intrinsic to the matrix, obstruct higher multiplicities for~$e_2$ and explain why only $e_1$ generically exhibits this phenomenon.

\subsection{Further directions}
Several directions remain open.
Beyond the finite discrete case treated here, it would be interesting to understand what replaces the witness-cell and order-chamber stratifications for general small $\Rbar$-categories, and whether the Events Theorem has an analogue in that setting.

Within the discrete setting, the chamberwise lattice towers suggest computable invariants of the nucleus that go beyond the cell structure alone.
As the basepoint moves, concepts are born and die at different event radii; tracking these births and deaths produces persistence-type summaries of the lattice changes.
In applications where the matrix $M$ arises from data, the lattice tower at a point provides a structured decomposition of the data into formal concepts at varying resolutions, with the Galois connections at each level encoding principled algebraic operations (joins and meets) on the resulting clusters.
Developing these invariants in examples is a natural next step.

A further direction concerns nuclei endowed with additional compatible structure.
When the profunctor is compatible with monoidal data, the nucleus inherits richer operations relevant to linear realizability; see \cite{Jarvis2025NucleusProfunctor,seiller-hdr,GastaldiJarvisSeillerTerillaLinearRealizability}.
The interaction between these monoidal structures and the projective metric geometry of the present paper remains to be explored.

\bibliography{biblio}


\begin{thebibliography}{30}
\ifx \bisbn   \undefined \def \bisbn  #1{ISBN #1}\fi
\ifx \binits  \undefined \def \binits#1{#1}\fi
\ifx \bauthor  \undefined \def \bauthor#1{#1}\fi
\ifx \batitle  \undefined \def \batitle#1{#1}\fi
\ifx \bjtitle  \undefined \def \bjtitle#1{#1}\fi
\ifx \bvolume  \undefined \def \bvolume#1{\textbf{#1}}\fi
\ifx \byear  \undefined \def \byear#1{#1}\fi
\ifx \bissue  \undefined \def \bissue#1{#1}\fi
\ifx \bfpage  \undefined \def \bfpage#1{#1}\fi
\ifx \blpage  \undefined \def \blpage #1{#1}\fi
\ifx \burl  \undefined \def \burl#1{\textsf{#1}}\fi
\ifx \doiurl  \undefined \def \doiurl#1{\url{https://doi.org/#1}}\fi
\ifx \betal  \undefined \def \betal{\textit{et al.}}\fi
\ifx \binstitute  \undefined \def \binstitute#1{#1}\fi
\ifx \binstitutionaled  \undefined \def \binstitutionaled#1{#1}\fi
\ifx \bctitle  \undefined \def \bctitle#1{#1}\fi
\ifx \beditor  \undefined \def \beditor#1{#1}\fi
\ifx \bpublisher  \undefined \def \bpublisher#1{#1}\fi
\ifx \bbtitle  \undefined \def \bbtitle#1{#1}\fi
\ifx \bedition  \undefined \def \bedition#1{#1}\fi
\ifx \bseriesno  \undefined \def \bseriesno#1{#1}\fi
\ifx \blocation  \undefined \def \blocation#1{#1}\fi
\ifx \bsertitle  \undefined \def \bsertitle#1{#1}\fi
\ifx \bsnm \undefined \def \bsnm#1{#1}\fi
\ifx \bsuffix \undefined \def \bsuffix#1{#1}\fi
\ifx \bparticle \undefined \def \bparticle#1{#1}\fi
\ifx \barticle \undefined \def \barticle#1{#1}\fi
\bibcommenthead
\ifx \bconfdate \undefined \def \bconfdate #1{#1}\fi
\ifx \botherref \undefined \def \botherref #1{#1}\fi
\ifx \url \undefined \def \url#1{\textsf{#1}}\fi
\ifx \bchapter \undefined \def \bchapter#1{#1}\fi
\ifx \bbook \undefined \def \bbook#1{#1}\fi
\ifx \bcomment \undefined \def \bcomment#1{#1}\fi
\ifx \oauthor \undefined \def \oauthor#1{#1}\fi
\ifx \citeauthoryear \undefined \def \citeauthoryear#1{#1}\fi
\ifx \endbibitem  \undefined \def \endbibitem {}\fi
\ifx \bconflocation  \undefined \def \bconflocation#1{#1}\fi
\ifx \arxivurl  \undefined \def \arxivurl#1{\textsf{#1}}\fi
\csname PreBibitemsHook\endcsname

\bibitem[\protect\citeauthoryear{Develin and
  Sturmfels}{2004}]{develinSturmfels2004tropical}
\begin{barticle}
\bauthor{\bsnm{Develin}, \binits{M.}},
\bauthor{\bsnm{Sturmfels}, \binits{B.}}:
\batitle{Tropical convexity}.
\bjtitle{Documenta Mathematica}
\bvolume{9},
\bfpage{1}--\blpage{27}
(\byear{2004})
\doiurl{10.4171/DM/154}
\end{barticle}
\endbibitem

\bibitem[\protect\citeauthoryear{Karp}{1978}]{Karp1978}
\begin{barticle}
\bauthor{\bsnm{Karp}, \binits{R.M.}}:
\batitle{A characterization of the minimum cycle mean in a digraph}.
\bjtitle{Discrete Mathematics}
\bvolume{23}(\bissue{3}),
\bfpage{309}--\blpage{311}
(\byear{1978})
\doiurl{10.1016/0012-365X(78)90011-0}
\end{barticle}
\endbibitem

\bibitem[\protect\citeauthoryear{Kuhn}{1955}]{Kuhn1955}
\begin{barticle}
\bauthor{\bsnm{Kuhn}, \binits{H.W.}}:
\batitle{The hungarian method for the assignment problem}.
\bjtitle{Naval Research Logistics Quarterly}
\bvolume{2}(\bissue{1--2}),
\bfpage{83}--\blpage{97}
(\byear{1955})
\doiurl{10.1002/nav.3800020109}
\end{barticle}
\endbibitem

\bibitem[\protect\citeauthoryear{Burkard
  et~al.}{2009}]{BurkardDellAmicoMartello2009}
\begin{bbook}
\bauthor{\bsnm{Burkard}, \binits{R.}},
\bauthor{\bsnm{Dell'Amico}, \binits{M.}},
\bauthor{\bsnm{Martello}, \binits{S.}}:
\bbtitle{Assignment Problems}.
\bpublisher{Society for Industrial and Applied Mathematics},
\blocation{Philadelphia}
(\byear{2009})
\end{bbook}
\endbibitem

\bibitem[\protect\citeauthoryear{Maclagan and
  Sturmfels}{2015}]{MaclaganSturmfels2015ITG}
\begin{bbook}
\bauthor{\bsnm{Maclagan}, \binits{D.}},
\bauthor{\bsnm{Sturmfels}, \binits{B.}}:
\bbtitle{Introduction to Tropical Geometry}.
\bsertitle{Graduate Studies in Mathematics},
vol. \bseriesno{161}.
\bpublisher{American Mathematical Society},
\blocation{Providence, RI}
(\byear{2015}).
\doiurl{10.1090/gsm/161}
\end{bbook}
\endbibitem

\bibitem[\protect\citeauthoryear{Gaubert and
  Katz}{2006}]{GaubertKatz2006MaxPlusConvexGeometry}
\begin{bchapter}
\bauthor{\bsnm{Gaubert}, \binits{S.}},
\bauthor{\bsnm{Katz}, \binits{R.}}:
\bctitle{Max-plus convex geometry}.
In: \bbtitle{Relations and Kleene Algebra in Computer Science (RelMiCS 2006)}.
\bsertitle{Lecture Notes in Computer Science},
vol. \bseriesno{4136},
pp. \bfpage{192}--\blpage{206}.
\bpublisher{Springer}, \blocation{???}
(\byear{2006}).
\doiurl{10.1007/11828563_13}
\end{bchapter}
\endbibitem

\bibitem[\protect\citeauthoryear{Gaubert and
  Katz}{2011}]{GaubertKatz2011MinimalHalfspaces}
\begin{barticle}
\bauthor{\bsnm{Gaubert}, \binits{S.}},
\bauthor{\bsnm{Katz}, \binits{R.}}:
\batitle{Minimal half-spaces and external representation of tropical
  polyhedra}.
\bjtitle{Journal of Algebraic Combinatorics}
\bvolume{33}(\bissue{3}),
\bfpage{325}--\blpage{348}
(\byear{2011})
\doiurl{10.1007/S10801-010-0246-4}
\end{barticle}
\endbibitem

\bibitem[\protect\citeauthoryear{Nitica† and Singer}{2007}]{Nitica01062007}
\begin{barticle}
\bauthor{\bsnm{Nitica†}, \binits{V.}},
\bauthor{\bsnm{Singer}, \binits{I.}}:
\batitle{Max-plus convex sets and max-plus semispaces. i}.
\bjtitle{Optimization}
\bvolume{56}(\bissue{1-2}),
\bfpage{171}--\blpage{205}
(\byear{2007})
\doiurl{10.1080/02331930600819852}
{\href{https://arxiv.org/abs/https://doi.org/10.1080/02331930600819852}{{https://doi.org/10.1080/02331930600819852}}}
\end{barticle}
\endbibitem

\bibitem[\protect\citeauthoryear{Gaubert and
  Sergeev}{2008}]{GaubertSergeev2013CyclicProjectors}
\begin{barticle}
\bauthor{\bsnm{Gaubert}, \binits{S.}},
\bauthor{\bsnm{Sergeev}, \binits{S.}}:
\batitle{Cyclic projectors and separation theorems in idempotent convex
  geometry}.
\bjtitle{Journal of Mathematical Sciences}
\bvolume{155}(\bissue{6}),
\bfpage{815}--\blpage{829}
(\byear{2008})
\doiurl{10.1007/s10958-008-9243-8} .
\bcomment{Translated from Fundamentalnaya i Prikladnaya Matematika}
\end{barticle}
\endbibitem

\bibitem[\protect\citeauthoryear{Lawvere}{1973}]{lawvere1973metric}
\begin{barticle}
\bauthor{\bsnm{Lawvere}, \binits{F.W.}}:
\batitle{Metric spaces, generalized logic, and closed categories}.
\bjtitle{Rendiconti del Seminario Matematico e Fisico di Milano}
\bvolume{43},
\bfpage{135}--\blpage{166}
(\byear{1973})
\doiurl{10.1007/BF02924844} .
\bcomment{Reprinted in Reprints in Theory and Applications of Categories, No.~1
  (2002), pp.~1--37}
\end{barticle}
\endbibitem

\bibitem[\protect\citeauthoryear{Isbell}{1960}]{Isbell1960Adequate}
\begin{barticle}
\bauthor{\bsnm{Isbell}, \binits{J.R.}}:
\batitle{Adequate subcategories}.
\bjtitle{Illinois Journal of Mathematics}
\bvolume{4}(\bissue{4}),
\bfpage{541}--\blpage{552}
(\byear{1960})
\doiurl{10.1215/ijm/1255456274}
\end{barticle}
\endbibitem

\bibitem[\protect\citeauthoryear{Avery and
  Leinster}{2021}]{averyLeinster2021isbell}
\begin{barticle}
\bauthor{\bsnm{Avery}, \binits{T.}},
\bauthor{\bsnm{Leinster}, \binits{T.}}:
\batitle{Isbell conjugacy and the reflexive completion}.
\bjtitle{Theory and Applications of Categories}
\bvolume{36}(\bissue{12}),
\bfpage{306}--\blpage{347}
(\byear{2021})
\doiurl{10.70930/tac/r1jknjot}
\end{barticle}
\endbibitem

\bibitem[\protect\citeauthoryear{Willerton}{2013}]{willerton2013tight}
\begin{barticle}
\bauthor{\bsnm{Willerton}, \binits{S.}}:
\batitle{Tight spans, {I}sbell completions and semi-tropical modules}.
\bjtitle{Theory and Applications of Categories}
\bvolume{28}(\bissue{22}),
\bfpage{696}--\blpage{732}
(\byear{2013})
\doiurl{10.70930/tac/rkp3zgxc}
\end{barticle}
\endbibitem

\bibitem[\protect\citeauthoryear{Willerton}{2014}]{willerton2014galois}
\begin{botherref}
\oauthor{\bsnm{Willerton}, \binits{S.}}:
Galois Correspondences and Enriched Adjunctions.
Blog post
(2014).
\url{https://golem.ph.utexas.edu/category/2014/02/galois_correspondences_and_enr.html}
Accessed 2025-12-08
\end{botherref}
\endbibitem

\bibitem[\protect\citeauthoryear{Willerton}{2015}]{willerton2015legendre}
\begin{barticle}
\bauthor{\bsnm{Willerton}, \binits{S.}}:
\batitle{The {L}egendre-{F}enchel transform from a category theoretic
  perspective}.
\bjtitle{arXiv e-prints}
(\byear{2015})
\doiurl{10.48550/arXiv.1501.03791}
{\href{https://arxiv.org/abs/1501.03791}{{arXiv:1501.03791}}}
{[math.CT]}
\end{barticle}
\endbibitem

\bibitem[\protect\citeauthoryear{Ambrosio and
  Gigli}{2013}]{ambrosioGigli2013usersGuide}
\begin{bchapter}
\bauthor{\bsnm{Ambrosio}, \binits{L.}},
\bauthor{\bsnm{Gigli}, \binits{N.}}:
\bctitle{A user's guide to optimal transport}.
In: \bbtitle{Modelling and Optimisation of Flows on Networks}.
\bsertitle{Lecture Notes in Mathematics},
vol. \bseriesno{2062},
pp. \bfpage{1}--\blpage{155}.
\bpublisher{Springer},
\blocation{Berlin, Heidelberg}
(\byear{2013}).
\doiurl{10.1007/978-3-642-32160-3_1}
\end{bchapter}
\endbibitem

\bibitem[\protect\citeauthoryear{Elliott}{2017}]{elliott2017fuzzy}
\begin{botherref}
\oauthor{\bsnm{Elliott}, \binits{J.A.}}:
On the fuzzy concept complex.
PhD thesis,
University of Sheffield,
Sheffield, UK
(2017).
\url{https://etheses.whiterose.ac.uk/18342/}
\end{botherref}
\endbibitem

\bibitem[\protect\citeauthoryear{Fujii}{2019}]{fujii2019enriched}
\begin{barticle}
\bauthor{\bsnm{Fujii}, \binits{S.}}:
\batitle{Enriched categories and tropical mathematics}.
\bjtitle{arXiv e-prints}
(\byear{2019})
\doiurl{10.48550/arXiv.1909.07620}
{\href{https://arxiv.org/abs/1909.07620}{{arXiv:1909.07620}}}
{[math.CT]}
\end{barticle}
\endbibitem

\bibitem[\protect\citeauthoryear{Bradley
  et~al.}{2022}]{BradleyTerillaVlassopoulos2022EnrichedLanguage}
\begin{barticle}
\bauthor{\bsnm{Bradley}, \binits{T.-D.}},
\bauthor{\bsnm{Terilla}, \binits{J.}},
\bauthor{\bsnm{Vlassopoulos}, \binits{Y.}}:
\batitle{An enriched category theory of language: From syntax to semantics}.
\bjtitle{La Matematica}
\bvolume{1}(\bissue{2}),
\bfpage{551}--\blpage{580}
(\byear{2022})
\doiurl{10.1007/s44007-022-00021-2}
\end{barticle}
\endbibitem

\bibitem[\protect\citeauthoryear{Gaubert and
  Vlassopoulos}{2024}]{GaubertVlassopoulos2024DirectedMetric}
\begin{barticle}
\bauthor{\bsnm{Gaubert}, \binits{S.}},
\bauthor{\bsnm{Vlassopoulos}, \binits{Y.}}:
\batitle{Directed metric structures arising in large language models}.
\bjtitle{arXiv e-prints}
(\byear{2024})
\doiurl{10.48550/arXiv.2405.12264}
{\href{https://arxiv.org/abs/2405.12264}{{arXiv:2405.12264}}}
{[cs.LG]}
\end{barticle}
\endbibitem

\bibitem[\protect\citeauthoryear{Bradley and
  Vigneaux}{2025}]{BradleyVigneaux2025MagnitudeTexts}
\begin{barticle}
\bauthor{\bsnm{Bradley}, \binits{T.-D.}},
\bauthor{\bsnm{Vigneaux}, \binits{J.P.}}:
\batitle{The magnitude of categories of texts enriched by language models}.
\bjtitle{Theory and Applications of Categories}
\bvolume{44}(\bissue{37}),
\bfpage{1256}--\blpage{1281}
(\byear{2025})
\doiurl{10.48550/arXiv.2501.06662}
{\href{https://arxiv.org/abs/2501.06662}{{arXiv:2501.06662}}}
{[math.CT]}
\end{barticle}
\endbibitem

\bibitem[\protect\citeauthoryear{Bradley
  et~al.}{2024}]{BradleyGastaldiTerilla2024}
\begin{barticle}
\bauthor{\bsnm{Bradley}, \binits{T.}},
\bauthor{\bsnm{Gastaldi}, \binits{J.L.}},
\bauthor{\bsnm{Terilla}, \binits{J.}}:
\batitle{The structure of meaning in language: Parallel narratives in linear
  algebra and category theory}.
\bjtitle{Notices of the American Mathematical Society}
\bvolume{71}(\bissue{2}),
\bfpage{174}--\blpage{185}
(\byear{2024})
\doiurl{10.1090/noti2868}
\end{barticle}
\endbibitem

\bibitem[\protect\citeauthoryear{Schrijver}{2003}]{Schrijver2003}
\begin{bbook}
\bauthor{\bsnm{Schrijver}, \binits{A.}}:
\bbtitle{Combinatorial Optimization: Polyhedra and Efficiency}.
\bsertitle{Algorithms and Combinatorics},
vol. \bseriesno{24}.
\bpublisher{Springer},
\blocation{Berlin}
(\byear{2003})
\end{bbook}
\endbibitem

\bibitem[\protect\citeauthoryear{Akian
  et~al.}{2023}]{AkianGaubertQiSaadi2023TropicalRegression}
\begin{barticle}
\bauthor{\bsnm{Akian}, \binits{M.}},
\bauthor{\bsnm{Gaubert}, \binits{S.}},
\bauthor{\bsnm{Qi}, \binits{Y.}},
\bauthor{\bsnm{Saadi}, \binits{O.}}:
\batitle{Tropical linear regression and mean payoff games: or, how to measure
  the distance to equilibria}.
\bjtitle{SIAM Journal on Discrete Mathematics}
\bvolume{37}(\bissue{2}),
\bfpage{632}--\blpage{674}
(\byear{2023})
\doiurl{10.1137/21M1428297}
\end{barticle}
\endbibitem

\bibitem[\protect\citeauthoryear{Jarvis}{2025}]{Jarvis2025NucleusProfunctor}
\begin{botherref}
\oauthor{\bsnm{Jarvis}, \binits{S.K.}}:
A novel closed monoidal structure on the nucleus of a profunctor.
PhD thesis,
The Graduate Center, City University of New York
(June 2025).
Doctoral dissertation (Ph.D.), Mathematics; advisor: John Terilla.
\url{https://academicworks.cuny.edu/gc_etds/6231}
\end{botherref}
\endbibitem

\bibitem[\protect\citeauthoryear{Seiller}{2024}]{seiller-hdr}
\begin{botherref}
\oauthor{\bsnm{Seiller}, \binits{T.}}:
Mathematical {I}nformatics.
PhD thesis,
Sorbonne Paris Nord University
(2024).
Habilitation thesis.
\url{https://theses.hal.science/tel-04616661}
\end{botherref}
\endbibitem

\bibitem[\protect\citeauthoryear{Gastaldi
  et~al.}{2025}]{GastaldiJarvisSeillerTerillaLinearRealizability}
\begin{botherref}
\oauthor{\bsnm{Gastaldi}, \binits{J.L.}},
\oauthor{\bsnm{Jarvis}, \binits{S.}},
\oauthor{\bsnm{Seiller}, \binits{T.}},
\oauthor{\bsnm{Terilla}, \binits{J.}}:
Linear realizability and structures in {$\mathbb{R}$}-enriched adjunctions.
Preprint, available from the authors
(2025)
\end{botherref}
\endbibitem

\bibitem[\protect\citeauthoryear{Kelly}{1982}]{Kelly1982}
\begin{bbook}
\bauthor{\bsnm{Kelly}, \binits{G.M.}}:
\bbtitle{Basic Concepts of Enriched Category Theory}.
\bsertitle{London Mathematical Society Lecture Note Series},
vol. \bseriesno{64}.
\bpublisher{Cambridge University Press}, \blocation{???}
(\byear{1982}).
\bcomment{Reprinted in: Reprints in Theory and Applications of Categories,
  No.~10, 2005}
\end{bbook}
\endbibitem

\bibitem[\protect\citeauthoryear{Aguiar and
  Mahajan}{2017}]{aguiarMahajan2017hyperplaneArrangements}
\begin{bbook}
\bauthor{\bsnm{Aguiar}, \binits{M.}},
\bauthor{\bsnm{Mahajan}, \binits{S.}}:
\bbtitle{Topics in Hyperplane Arrangements}.
\bsertitle{Mathematical Surveys and Monographs},
vol. \bseriesno{226}.
\bpublisher{American Mathematical Society},
\blocation{Providence, RI}
(\byear{2017})
\end{bbook}
\endbibitem

\bibitem[\protect\citeauthoryear{Wille}{1982}]{Wille1982}
\begin{bchapter}
\bauthor{\bsnm{Wille}, \binits{R.}}:
\bctitle{Restructuring lattice theory: An approach based on hierarchies of
  concepts}.
In: \beditor{\bsnm{Rival}, \binits{I.}} (ed.)
\bbtitle{Ordered Sets: Proceedings of the NATO Advanced Study Institute Held at
  Banff, Canada, August 28 to September 12, 1981}.
\bsertitle{NATO Advanced Study Institutes Series},
vol. \bseriesno{83},
pp. \bfpage{445}--\blpage{470}.
\bpublisher{Springer},
\blocation{Dordrecht}
(\byear{1982}).
\doiurl{10.1007/978-94-009-7798-3_15}
\end{bchapter}
\endbibitem

\end{thebibliography}

\end{document}